\documentclass[11pt,leqno]{amsart}

\usepackage[english]{babel}
\usepackage[letterpaper,top=2cm,bottom=2cm,left=1.5cm,right=1.5cm]{geometry}
\usepackage[T1]{fontenc}
\usepackage{mathrsfs}
\usepackage{booktabs} 
\usepackage{here}
\usepackage[utf8]{inputenc}
\usepackage{amsmath,amssymb,amsthm,amsfonts,dsfont,graphicx,enumitem}
\usepackage[colorlinks=true,linkcolor=blue,citecolor=blue,urlcolor=blue]{hyperref} 
\usepackage{mathtools}
\usepackage{tikz}
\usetikzlibrary{calc,intersections,decorations.pathreplacing}
\usepackage{tikz-3dplot}
\usepackage{pgfplots}
\usepackage{cite}
\usepackage[colorlinks=true,linkcolor=blue,citecolor=blue,urlcolor=blue]{hyperref}

\pgfplotsset{compat=1.16}

\usepackage{algpseudocode}
\usepackage{caption}
\usepackage{subcaption}

\DeclareCaptionType{algorithm}[Algorithm][List of Algorithms]
\usepackage{framed}

\theoremstyle{plain}
\newtheorem{theorem}{Theorem}[section]
\newtheorem{lemma}[theorem]{Lemma}

\theoremstyle{definition}

\newtheorem{remark}[theorem]{Remark}
\newtheorem{conj}[theorem]{Conjecture}

\newtheoremstyle{nonumwithname}
  {3pt}{3pt}
  {\normalfont}
  {}
  {\bfseries}
  {.}
  { }
  {\thmname{#1}\ \textsc{#3}}
\theoremstyle{nonumwithname}

\newcommand{\CR}{{\mbox{\tiny CR}}}
\newcommand{\CG}{{\mbox{\tiny CG}}}

\newcommand{\ip}[2]{\bigl(#1,\,#2\bigr)}

\newcommand{\R}{\mathbb{R}}

\begin{document}

\title[]
{
Sharp Dirichlet Eigenvalue Inequalities on Triangles
}

\author[Endo]{Ryoki Endo}
	\address{Faculty of Science\\
    Niigata University\\
    Niigata, Japan. JSPS Research Fellow (PD).
		8050, Ikarashi 2-no-cho\\
		Nishi-ku, Niigata 950-2181, Japan}
	\email{endo@m.sc.niigata-u.ac.jp}

\author[Liu]{Xuefeng Liu}
	\address{School of Arts and Sciences\\
		Tokyo Woman's Christian University\\
		2-6-1 Zenpukuji\\
		Suginami-ku, Tokyo 167-8585, Japan}
	\email{xfliu@lab.twcu.ac.jp}

    \author[Mariano]{Phanuel Mariano}
	\address{Department of Mathematics\\
		Union College\\
		Schenectady, NY 12308,  U.S.A.}
	\email{marianop@union.edu}

\maketitle

\begin{abstract}
	We prove sharp Dirichlet eigenvalue inequalities for planar triangles. We settle a conjecture of Laugesen and Siudeja by showing that the equilateral triangle uniquely minimizes a scale-invariant functional of the first Dirichlet eigenvalue, area, and perimeter. Consequences include an optimal two-term lower bound for the first Dirichlet eigenvalue in terms of area and perimeter. We also prove a Cheeger-type inequality with an explicit best constant considered by Parini. To prove these conjectures we propose a new method for proving Dirichlet eigenvalue inequalities on triangles. Our method is based on a new computable lower bound for second-order directional shape derivatives under vertex perturbations. It also uses validated finite-element error estimates and recently developed analytic estimates for eigenvalues of nearly degenerate triangles. The method is not specific to the functionals considered in
  this paper and  it can be used to prove various other eigenvalue inequalities on triangles.
\end{abstract}

\tableofcontents

\section{Introduction}

Let $\triangle\subset\mathbb{R}^2$ be a triangle, and consider the Dirichlet eigenvalue problem
\[
-\Delta u=\lambda u \quad \text{in }\triangle,
\qquad
u=0 \quad \text{on }\partial\triangle.
\]
Its spectrum consists of a discrete sequence
\[
0<\lambda_1(\triangle)<\lambda_2(\triangle)\le \cdots \to +\infty,
\]
where $\lambda_1(\triangle)$ denotes the first Dirichlet eigenvalue.

The first Dirichlet eigenvalue of the Laplacian is a central object in spectral shape optimization.
Among all planar domains of prescribed area, the disk uniquely minimizes $\lambda_1$ by the Faber--Krahn inequality.
A natural polygonal analogue, going back to Pólya--Szegő \cite{Polya-Szego-1951} and Pólya \cite{Polya-1948b}, asks whether, among all $n$-gons of fixed area, the regular $n$-gon minimizes the first Dirichlet eigenvalue.
This problem is classical and deceptively simple: the cases $n=3,4$ are known, whereas for general $n\ge 5$ the full polygonal problem remains highly nontrivial; see, for example, \cite{Laugesen-Siudeja-BookChapter-2017,Bogosel-Bucur-2024}.

Within the class of triangles, a number of sharp spectral inequalities are known.
P\'olya--Szeg\"o's classical result implies that the equilateral triangle minimizes $\lambda_1$ among triangles of fixed area
\begin{equation}
\lambda_{1}\left(\triangle\right)\geq\frac{4\pi^{2}}{\sqrt{3}\left|\triangle\right|}.\label{eq:Triang-Faber-Khran}
\end{equation}
Subsequent works established further sharp bounds involving the perimeter, the diameter, and eigenvalue sums; see \cite{Freitas-2007,Freitas-Siudeja-2010,Laugesen-Siudeja-2010,Laugesen-Siudeja-BookChapter-2017,Siudeja-2007,siudeja2010isoperimetric,brooks1987first,laugesen2010minimizing,laugesen2009maximizing}. Other isoperimetric-type inequalities on triangles are known such as the sharp fundamental gap \cite{Lu-Rowlett-2013}, a sharp Ashbaugh-Benguria-Payne-Pólya-Weinberger inequality \cite{siudeja2010isoperimetric,Arbon-etall-2022}, sharp inequalities for mixed eigenvalues \cite{Siudeja-mixed-2016,Chen-Hongbin-Changfeng-2026} and sharp lower bound between the principal frequency and the torsional rigidity \cite{Banuelos-Mariano-2024}. In general, proving spectral properties for triangles can be quite difficult, as illustrated by problems such as the Hot Spots problem \cite{Siudeja-HotSpots-2015,Judge-Sugata-2020} that took several years to solve. There are still many interesting open problems in the spectral theory of triangles for which new tools are needed.

In \cite{Makai-1962}, Makai proved that for every convex domain  $\Omega \subset \mathbb{R}^2$,
\begin{equation}
\lambda_{1}(\Omega)\geq \frac{\pi^{2}}{16}\frac{|\partial\Omega|^{2}}{|\Omega|^{2}}.
\label{eq:Makai}
\end{equation}
In \cite{Siudeja-2007}, Siudeja showed that this inequality is asymptotically attained by thinning triangles and gave a sharp quantative upper bound.
Motivated by the numerical evidence of Antunes and Freitas \cite{antunes2006new}, Laugesen and Siudeja formulated a conjectural refinement for the first Dirichlet eigenvalue of triangles that combines area and perimeter in a sharp scale-invariant form \cite[Conjecture~6.7]{Laugesen-Siudeja-BookChapter-2017}.

\begin{conj}[Conjecture 6.7 of \cite{Laugesen-Siudeja-BookChapter-2017}]\label{Conj:1}
    The functional
    \begin{equation}\label{eq:Conjecture}
        \mathcal{F}(\triangle)
        :=
        \lambda_1(\triangle)|\triangle|
        -\frac{\pi^2}{16}\frac{|\partial\triangle|^2}{|\triangle|}
    \end{equation}
    is minimized uniquely by the equilateral triangle, and the minimum value is
    $7\sqrt{3}\pi^2/12$.
\end{conj}

Conjecture~\ref{Conj:1} immediately implies the following sharp two-term
lower bound, which is the lower-bound part of the two-sided estimate
conjectured by Siudeja \cite[Conjecture~1.3]{Siudeja-2007}.
\begin{conj}[Lower-bound part of Conjecture 1.3 in \cite{Siudeja-2007}]\label{Conj:2}
    For any triangle $\triangle$,
    \begin{equation}
        \lambda_{1}(\triangle)
        \ge
        \frac{\pi^{2}|\partial\triangle|^{2}}{16|\triangle|^{2}}
        +\frac{7\sqrt{3}\pi^{2}}{12|\triangle|},
    \end{equation}
    where equality holds if and only if $\triangle$ is equilateral.
\end{conj}

These conjectures are sharp quantitative versions of Makai's inequality \eqref{eq:Makai} for triangles. The paper of Antunes and Freitas \cite{antunes2006new}, originally conjectured that there exists some constant $\theta_1$ such that for all $\theta \leq \theta_1$, one has 
\begin{equation}
\lambda_{1}\left(\triangle\right)\geq\frac{4\pi^{2}}{\sqrt{3}\left|\triangle\right|}+\theta\frac{\left|\partial\triangle\right|^{2}-12\sqrt{3}\left|\triangle\right|}{\left|\triangle\right|^{2}}.\label{eq:Antunes-Freitas}
\end{equation}
Note that this is a quantitative version of P\'olya--Szeg\"o's polygonal Faber–Krahn inequality of \eqref{eq:Triang-Faber-Khran}. Inequality \eqref{eq:Antunes-Freitas} is equivalent to Conjectures \ref{Conj:1} and \ref{Conj:2} when $\theta_1=\frac{\pi^2}{16}$. These inequalities were also studied by Indrei in \cite[Corollary I.2.]{Indrei-2024} where \eqref{eq:Antunes-Freitas} was proved for  some computable $\theta_1$ that satisfies $\theta_1\leq \frac{\pi^2}{16}$. Indrei also showed that this inequality implies a quantitative version of triangular Faber-Krahn inequality with a sharp exponent in terms of the Fraenkel asymmetry \cite[Corollary I.4.]{Indrei-2024}. The sharp quantitative form of the Faber-Krahn inequality over general open sets was solved in \cite{brasco2015faber}. Similar quantitative versions were also studied recently by \cite{Amato-optimal-2025} for general convex domains.

Another inequality we will study is the well-known result by J. Cheeger \cite{Cheeger-1969}  that states that every bounded domain
$\Omega\subset\mathbb{R}^{n}$ satisfies 
\begin{equation}
\frac{\lambda_{1}\left(\Omega\right)}{h\left(\Omega\right)^{2}}\geq\frac{1}{4},\label{eq:Cheeger}
\end{equation}
where the Cheeger constant is defined by 
\begin{equation}
h\left(\Omega\right)=\inf\left\{ \frac{\left|\partial E\right|}{\left|E\right|}:E\subset\Omega\right\} .\label{eq:defn:Cheeger}
\end{equation}
Recall that, for a triangle $\triangle$, the Cheeger constant is given by
\[
h(\triangle)
=
\frac{|\partial\triangle|+\sqrt{4\pi|\triangle|}}{2|\triangle|}.
\]
For this fact, see \cite{kawohl2006characterization,brooks1987first}.
In \cite[Sec 6]{Parini-2017} it was conjectured that, among all convex $n$-gons
with a fixed number of edges, the Cheeger functional in \eqref{eq:Cheeger}
is minimized by the regular polygon for $n=3,4$. For triangles, this
conjecture can be written in the following way:

\begin{conj}[Cheeger-type inequality \cite{Parini-2017}]\label{Conj:3}
    For any triangle $\triangle$,
    \[
    \frac{\lambda_{1}(\triangle)}{h(\triangle)^{2}}
    \ge
    \frac{4\pi^{2}}{\left(3+\sqrt{\pi\sqrt{3}}\right)^{2}}\approx1.3885,
    \]
    where equality holds if and only if $\triangle$ is equilateral.
\end{conj}
This conjecture is motivated by the original paper by Parini
\cite[Sec 6]{Parini-2017} where he proves the existence of a minimizer for planar convex sets and improves the Cheeger inequality  \eqref{eq:Cheeger} to $$\frac{\lambda_{1}\left(\Omega\right)}{h\left(\Omega\right)^{2}}\geq\frac{\pi^{2}}{16}\approx 0.616,$$
for planar convex sets $\Omega$. In \cite[Prop. 3.1]{Ftouhi2021}, this lower bound is improved to $$\frac{\lambda_{1}(\triangle)}{h(\triangle)^{2}}\ge 0.902...$$ in the case of triangles. Parini conjectures that the square is the minimizer among planar convex sets, that is \begin{equation}
\frac{\lambda_{1}\left(\Omega\right)}{h\left(\Omega\right)^{2}}\geq\frac{\lambda_{1}\left(\left(0,1\right)^{2}\right)}{h\left(\left(0,1\right)^{2}\right)^{2}}=\frac{2\pi^{2}}{\left(2+\sqrt{\pi}\right)^{2}}\approx1.387.\label{eq:Cheeger-conj}
\end{equation}
We note that if \eqref{eq:Cheeger-conj} is true, then among $n$-sided polygons, the regular $n$-gon cannot be the minimizer for this functional when $n\geq 5$. In this case we would expect that the degenerating square would be the minimizer for all $n$-sided n-gons when $n\geq 5$.  

The Cheeger inequality was also studied in higher dimensions in the works of \cite{Ftouhi2021,Pratelli-Saracco-2025}. We point to the works of \cite{Briani-Buttazzo-Prinari-2023,Brasco-2022} for other generalizations of these results to the $p$-Laplacian and other functionals. Recently in \cite{bucur2025optimal}, the existence of a minimizer for the Cheeger functional is considered among larger classes of sets other than the convex ones. We also note that Brooks and Waksman in \cite{brooks1987first} originally considered the non-sharp Cheeger inequality on triangles. 

Conjecture~\ref{Conj:3} is stronger than the triangular case of P\'olya 's polygonal Faber-Krahn inequality.
Indeed, for triangles the Cheeger constant admits an explicit expression in terms of the area and the perimeter, and hence, under a fixed area constraint, minimizing $h(\triangle)$ is equivalent to minimizing $|\partial\triangle|$.
Therefore, the equilateral triangle is also the unique minimizer of the Cheeger constant among triangles of prescribed area.
Consequently, if Conjecture~\ref{Conj:3} holds, then the equilateral triangle must also minimize $\lambda_1$ among triangles of fixed area.

\medskip

Recently, there has been an increasing
interest in proving inequalities using computer assisted proofs through rigorous verified computing. As computing power grows and verified computation continues to advance, we can expect many more results to be established through computer-assisted proofs. We point to the recent works of 
\cite{Endo-Liu-2025-JDE-degenerate,endo2026-second-non-equilateral} where the simplicity of the second Dirichlet eigenvalue for triangles was established. We also point to the works of 
\cite{endo2025stable,endo2023shape,gomez2021any} for other results where verified computation is used in the spectral geometry of triangles. Verified computation was also used recently in verifying the local minimality of the polygonal Faber–Krahn inequality for $n=5,6$ conjectured by P\'olya \cite{Bogosel-Bucur-2024,bogosel-bucur-localmin-2024}. Elements of computer assisted proofs were used in the proof of Polya's conjecture on the Weyl bounds for 
the eigenvalue counting functions of the Dirichlet and Neumann Laplacian for the ball and annuli \cite{filonov2023polya,filonov2026polya}. 
We point the reader to \cite{nigam2025intersection} for a comprehensive treatment for results proven using verified computational techniques. 

\medskip

In this paper, we prove Conjectures~\ref{Conj:1}, \ref{Conj:2}, and~\ref{Conj:3} by a computer-assisted proof based on rigorous verified computation.
Our approach combines explicit shape derivative formulas for vertex perturbations of triangles, validated finite element estimates, and analytic bounds for nearly degenerate triangles.
The central analytic ingredient is a computable lower bound for second-order directional shape derivatives of Dirichlet eigenvalues.
Starting from the first- and second-order shape derivative formulas for nonsmooth domains due to Laurain \cite{Laurain-2020}, we derive a spectral representation of the second-order derivative and convert it into a rigorous computable lower bound by means of a truncation argument, verified eigenvalue bounds, and eigenspace error estimates.
To compute the concrete error bounds, we use the fully computable eigenspace estimates of Liu--Vejchodsk\'y \cite{liu2022fully}.

\medskip

The main result of this paper is the following.

\begin{theorem}
\label{thm:main-problem}
For every triangle $\triangle$, the shape functionals
\begin{align*}
J_1(\triangle)
&:= \lambda_1(\triangle)\,|\triangle|
   - \frac{\pi^{2}}{16}\,\frac{|\partial\triangle|^{2}}{|\triangle|}
   - \frac{7\sqrt{3}\,\pi^{2}}{12},\\
J_2(\triangle)
&:= \lambda_1(\triangle)\,|\triangle|
   - \frac{4\pi^2}{\left(3+\sqrt{\pi\sqrt{3}}\right)^2}
     \cdot
     \frac{\left(|\partial\triangle|+\sqrt{4\pi|\triangle|}\right)^2}{4|\triangle|}
\end{align*}
satisfy
\[
J_1(\triangle)\ge0,
\qquad
J_2(\triangle)\ge0.
\]
In each case, equality holds if and only if $\triangle$ is equilateral.
\end{theorem}

Theorem~\ref{thm:main-problem} settles the lower bound conjecture of Laugesen and Siudeja, implies the sharp two-term lower bound conjectured by Siudeja, and proves the Cheeger-type inequality conjectured by Parini.

\medskip

We next describe the method of proof.  Let $\triangle^p$ denote the triangle with vertices $(0,0)$, $(1,0)$, and $p=(x,y)$, and let $e=(a,b)\in\mathbb{R}^2$ be a direction of perturbation of the third vertex.
We write $
\lambda_i^p:=\lambda_i(\triangle^p)$.
Assuming that $\lambda_i^{p+te}$ is simple for $t$ in a neighborhood of zero, we define
\[
\dot{\lambda}_{i}^{p}
:=
\lim_{t\to0}
\frac{\lambda_i^{p+te}-\lambda_i^p}{t},
\qquad
\ddot{\lambda}_{i}^{p}
:=
\lim_{t\to0}
\frac{\lambda_i^{p+te}+\lambda_i^{p-te}-2\lambda_i^p}{t^2}.
\]
The computable estimate used in the proof may be stated, in simplified form,
as follows.

\begingroup
\renewcommand{\thetheorem}{\ref{thm:main-theorem-est}}
\begin{theorem}
There exists a computable quantity
$\underline{\ddot\lambda}^{\,p}_{i}$, defined explicitly in
\eqref{eq:main-theorem-est}, such that
\begin{equation}\label{eq:lower-bound-ddot}
\ddot{\lambda}_{i}^{p}
\ge
\underline{\ddot\lambda}^{\,p}_{i}.
\end{equation}
\end{theorem}
\endgroup
We apply this estimate to show the local optimality of the two scale-invariant functionals $J_1$ and $J_2$.
\medskip

By scaling, rotation, and reflection, every triangle may be represented by
\[
\triangle^p=\operatorname{conv}\{(0,0),(1,0),p\},
\qquad
p=(x,y)\in\Omega,
\]
where $\operatorname{conv}(E)\subset \mathbb{R}^2$ is the convex hull of $E\subset \mathbb{R}^2$ and 
\[
\Omega
:=
\left\{
(x,y)\in\mathbb{R}^2:
x^2+y^2\le 1,\ x\ge \frac12,\ y>0
\right\}.
\]
The equilateral triangle corresponds to
\[
p_0:=\left(\frac12,\frac{\sqrt3}{2}\right).
\]
For the computer-assisted proof, we consider the following decomposition of the shape space $\Omega$:
\[
\Omega=\Omega_{\mathrm{up}}\cup\Omega_{\mathrm{mid}}\cup\Omega_{\mathrm{down}},
\]
where
\begin{align*}
\Omega_{\mathrm{up}}
&:=
\left\{
(x,y)\in\Omega:
y\ge \frac{\sqrt3}{2}-\varepsilon_{\mathrm{up}}
\right\},\\
\Omega_{\mathrm{mid}}
&:=
\left\{
(x,y)\in\Omega:
\varepsilon_{\mathrm{down}}\le y,\ 
(x-\tfrac12)^2+(y-\tfrac{\sqrt3}{2})^2
\ge \varepsilon_{\mathrm{up}}^2
\right\},\\
\Omega_{\mathrm{down}}
&:=
\left\{
(x,y)\in\Omega:
y<\varepsilon_{\mathrm{down}}
\right\},
\end{align*}
with the values of 
$
\varepsilon_{\mathrm{up}}$ and  $
\varepsilon_{\mathrm{down}}
$ decided in \eqref{eq:epsilon_up_down} in the computer-assisted proof. 
These subregions and the contour plots of $J_1$ and $J_2$ are illustrated in
Figures~\ref{fig:omega-subregions} and~\ref{fig:J1J2-contours}, respectively.

\begin{figure}[H]
    \centering

    \begin{subfigure}[t]{0.47\textwidth}
        \centering
        \resizebox{\linewidth}{!}{%
            \begin{tikzpicture}[scale=6,>=stealth]

  \def\eps{0.3} 

  \coordinate (O) at (0,0);
  \coordinate (A) at (0,0);                     
  \coordinate (B) at (1,0);                     
  \coordinate (C) at ({1/2},{sqrt(3)/2});       
  \coordinate (X) at ({1/2},0);                 
  \coordinate (P) at ($(C)+(0,-\eps)$);         

  \coordinate (R) at ({atan(sqrt(3) - 2*\eps)}:1);

  \draw (A) -- (B);
  \draw[name path=outer] (B) arc[start angle=0,end angle=60,radius=1];
  \draw (X) -- (C);


  
\draw[dash pattern=on 1pt off 1pt] (0.5,0.15) -- ({sqrt(1-0.15*0.15)},0.15);

  \draw[dashed]
  (P) -- ({sqrt(1-(sqrt(3)/2-\eps)*(sqrt(3)/2-\eps))},
          {sqrt(3)/2-\eps});

  
\draw[dash pattern=on 1pt off 1pt, name path=inner]
    (C) ++(-90:\eps) arc[start angle=-90,end angle=-40,radius=\eps];

  \draw[decorate,decoration={brace,mirror,raise=2pt}]
    (C) -- (P) node[midway,left=4pt] {$\varepsilon_{\mathrm{up}}$};

\draw[decorate,decoration={brace,raise=2pt}]
    (X) -- (0.5,0.15) node[midway,left=4pt] {$\varepsilon_{\mathrm{down}}$};

  \node at (0.59,0.70) {$\Omega_{\text{up}}$};
  \node at (0.70,0.34) {$\Omega_{\text{mid}}$};
  \node at (0.80,0.08) {$\Omega_{\text{down}}$};

  \node[below left] at (A) {$(0,0)$};
  \node[below]      at (B) {$(1,0)$};
  \node[above]      at (C) {$\left(\tfrac12,\tfrac{\sqrt3}{2}\right)$};

\end{tikzpicture}
        }
        \caption{Subregions $\Omega_{\mathrm{up}}$, $\Omega_{\mathrm{mid}}$, and
        $\Omega_{\mathrm{down}}$}
        \label{fig:omega-subregions}
    \end{subfigure}
    \hspace{0.02\textwidth}
    \begin{subfigure}[t]{0.47\textwidth}
        \centering
        \includegraphics[width=\linewidth]{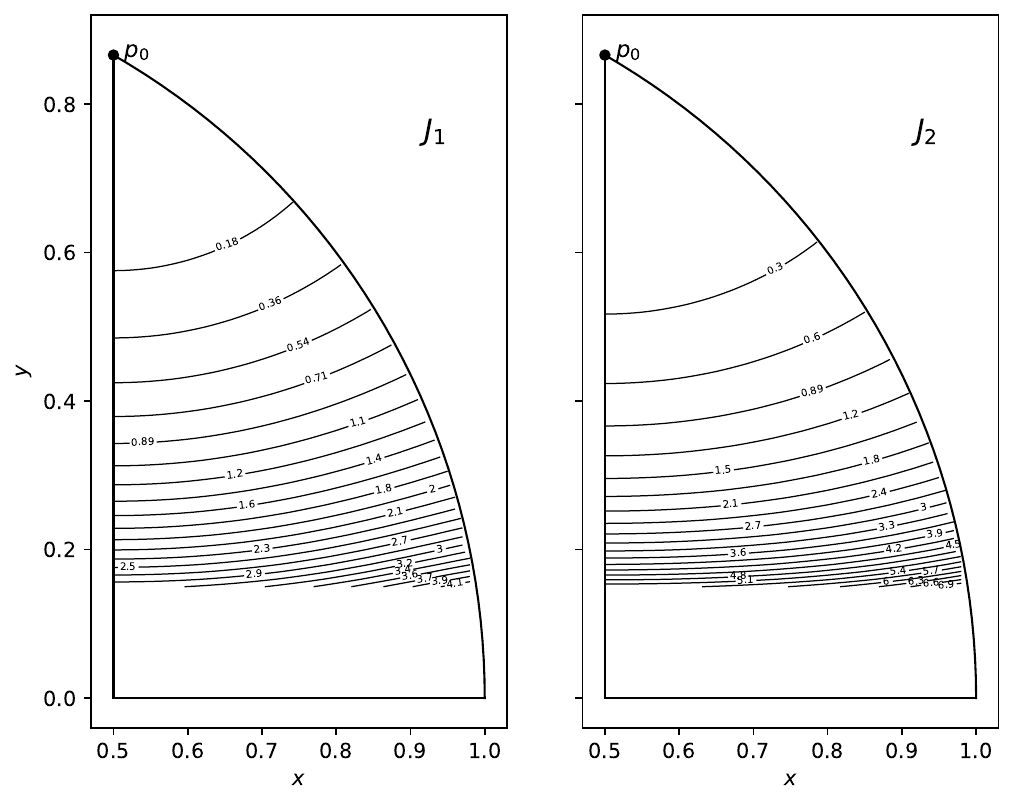}
        \caption{Contour plots of $J_1$ and $J_2$ ~(left: $J_1$, right: $J_2$)}
        \label{fig:J1J2-contours}
    \end{subfigure}

    \caption{Decomposition of the shape space and contour plots of the two
    scale-invariant functionals.}
    \label{fig:omega-intro}
\end{figure}
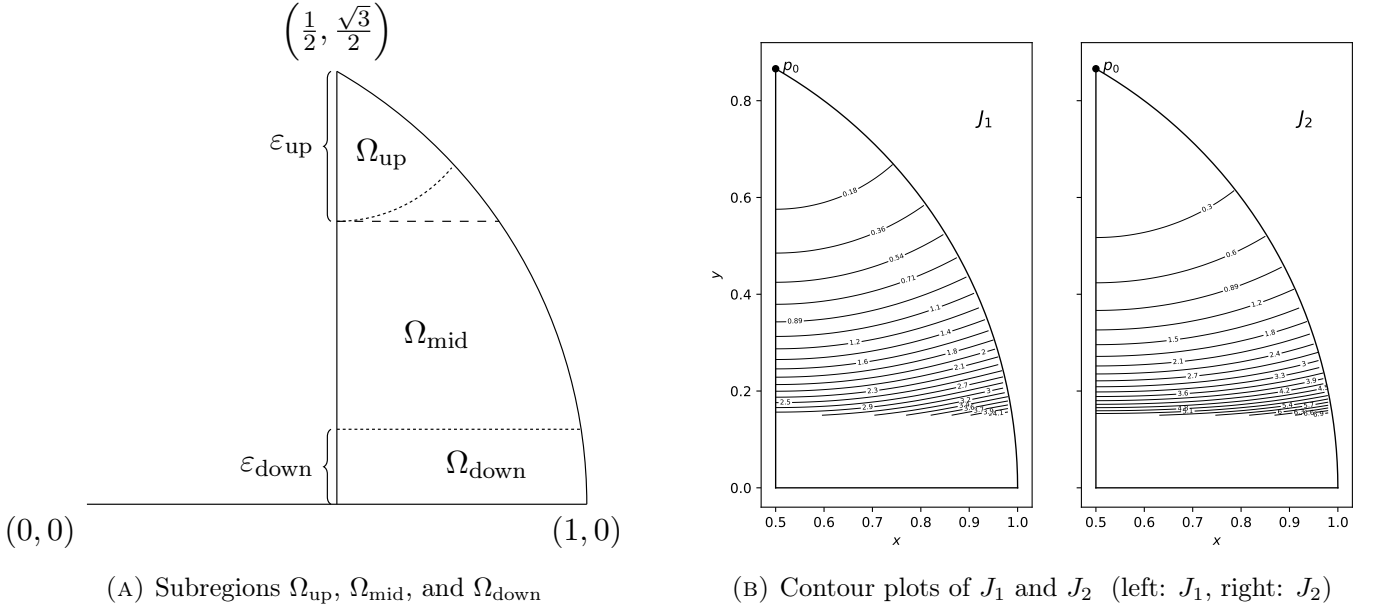

The proof of Theorem~\ref{thm:main-problem} is divided into three parts.

\begin{enumerate}[label=\emph{(\roman*)},leftmargin=*]
\item \textbf{The region $\Omega_{\mathrm{up}}$.}
First, we prove analytically that
\[
\nabla J_k(\triangle^{p_0})=0
\qquad (k=1,2).
\]
We then use the computable lower bound \eqref{eq:lower-bound-ddot} to verify 
\[
\frac{\partial^2 J_k}{\partial x^2}>0
\qquad (k=1,2),
\]
over $\Omega_{\mathrm{up}}$.
Since $J_k$ is symmetric with respect to reflection across the axis $x=1/2$, this convexity implies that any minimizer of $J_k$ in $\Omega_{\mathrm{up}}$ must lie on the symmetry axis $x=1/2$.
On this axis, we further verify
\[
\frac{\partial^2 J_k}{\partial y^2}>0
\qquad (k=1,2).
\]
It follows that $p_0$ is the unique minimizer of both $J_1$ and $J_2$ in $\Omega_{\mathrm{up}}$. See Section \ref{sec:case1}.

\item \textbf{The region $\Omega_{\mathrm{mid}}$.}
The region $\Omega_{\mathrm{mid}}$ is covered by finitely many verification cells $\{\mathcal{C}_{ij}\}$.
On each cell, the domain monotonicity of Dirichlet eigenvalues, together with rigorous eigenvalue enclosures at selected vertices of the cell, yields uniform lower bounds for $J_1$ and $J_2$.
The certified computation gives
\[
J_1(\triangle^p)\ge 6.85\cdot 10^{-7},
\qquad
J_2(\triangle^p)\ge 1.23\cdot 10^{-5}
\qquad
\text{for all }p\in\Omega_{\mathrm{mid}}.
\]
Thus neither functional can attain its global minimum in $\Omega_{\mathrm{mid}}$. See Section \ref{sec:case2}.

\item \textbf{The region $\Omega_{\mathrm{down}}$.}
In the nearly degenerate regime $y\to0$, the first Dirichlet eigenvalue satisfies
\[
\lambda_1(\triangle^p)\sim \frac{\pi^2}{y^2}.
\]
Direct finite element verification is therefore inefficient, and we treat
$J_1$ and $J_2$ separately.

For $J_1$, the leading term $\pi^2/y^2$ cancels with the leading perimeter
contribution.  We therefore use the lower bound of
Theorem~3.5 in \cite{Endo-Liu-2025-JDE-degenerate} for thin triangles: for
$\widetilde{\triangle}^{(s,t)}=\operatorname{conv}\{(-1,0),(1,0),(s,t)\}$
and $t\in(0,t_0]$,
\[
  t^{4/3}\left(
  \lambda_k\bigl(\widetilde{\triangle}^{(s,t)}\bigr)
  -\frac{\pi^2}{t^2}
  \right)
  \ge
  \frac{(2\pi^2)^{2/3}\kappa_k(s)}
  {1+\frac{t_0^{2/3}}{3\pi^2}(2\pi^2)^{2/3}\kappa_k(s)},
\]
where $\kappa_k(s)$ is the $k$-th positive root of
\[
  \sqrt[3]{1+s}\,
  \mathcal{A}\bigl((1+s)^{2/3}\kappa\bigr)
  \mathcal{A}'\bigl((1-s)^{2/3}\kappa\bigr)
  +
  \sqrt[3]{1-s}\,
  \mathcal{A}\bigl((1-s)^{2/3}\kappa\bigr)
  \mathcal{A}'\bigl((1+s)^{2/3}\kappa\bigr)=0
\]
with $\mathcal{A}(u):=\mathrm{Ai}(-u)$.  This gives the required
Airy-type correction to the leading singular term and yields
\[
J_1(\triangle^p)>0
\qquad
\text{for all }p\in\Omega_{\mathrm{down}}.
\]

For $J_2$, we use the lower bound of
Freitas--Siudeja~\cite[Corollary~4.1]{Freitas-Siudeja-2010}.  In the present
normalization, this estimate gives
\[
\lambda_1(\triangle^p)
\ge
\frac{\pi^2(1+y)^2}{y^2}.
\]
Together with elementary geometric estimates, it follows
that
\[
J_2(\triangle^p)>0
\qquad
\text{for all }p\in\Omega_{\mathrm{down}}.
\]
See Section \ref{sec:case3}.
\end{enumerate}

The verifications in $\Omega_{\mathrm{up}}$ and $\Omega_{\mathrm{mid}}$ are implemented in Algorithms~\ref{alg:eigen-errors}--\ref{algorithm-3}.
These algorithms combine validated finite element eigenvalue estimates, eigenvalue perturbation bounds, and the computable eigenspace error estimates of Liu--Vejchodsk\'y~\cite{liu2022fully}.
Combining the three estimates above, we conclude that $p_0$ is the unique global minimizer of $J_k$ on $\Omega$ for $k=1,2$.

The method is not specific to the functionals $J_1$ and $J_2$ considered in
this paper. In principle, the same strategy can be applied to other geometric
scale-invariant Dirichlet eigenvalue functionals on triangles with suitable choices of $
\varepsilon_{\mathrm{up}}$ and  $
\varepsilon_{\mathrm{down}}
$, provided that
one can obtain similar analytic control in the nearly degenerate regime.


The code used in the computer-assisted proof is available at
\begin{center}
\url{https://github.com/ryendo/LowerBoundsIneq}.
\end{center}

\subsection*{Structure of the paper}

The remainder of the paper is organized as follows.
In Section~\ref{Preliminaries}, we introduce the notation and recall explicit first- and second-order directional derivative formulas for Dirichlet eigenvalues on triangles.
In Section~\ref{eq:est-drivatives}, we derive a computable lower bound for the second-order directional derivative.
Finally, in Section~\ref{Sec:Proof-of-Lower}, we complete the computer-assisted proof of the main theorem.

\section{Preliminaries}\label{Preliminaries}

Let $\triangle \subset \mathbb{R}^2$ be a triangular domain. We use the standard notation for Sobolev spaces.  
The space $L^2(\triangle)$ denotes the real Hilbert space of square-integrable functions on $\triangle$, equipped with the inner product and norm
\[
  (u,v)_\triangle := \int_\triangle uv\,dx,
  \qquad
  \|v\|_\triangle := (v,v)_\triangle^{1/2}.
\]
We denote by $H^1(\triangle)$ the Sobolev space of functions in $L^2(\triangle)$ whose weak first derivatives also belong to $L^2(\triangle)$.  
Moreover, $H^1_0(\triangle)$ denotes the closure of $C_0^\infty(\triangle)$ in $H^1(\triangle)$.  
Since functions in $H^1_0(\triangle)$ vanish on $\partial\triangle$, the bilinear form
\[
  (u,v)_{H^1_0(\triangle)} := (\nabla u,\nabla v)_\triangle
  \qquad \text{for } u,v \in H^1_0(\triangle)
\]
defines an inner product on $H^1_0(\triangle)$.

The weak formulation of the Dirichlet eigenvalue problem for the Laplacian is to find $(\lambda,u)\in (0,\infty)\times (H^1_0(\triangle)\setminus\{0\})$ such that
\begin{equation}\label{eq:eigenvalue-problem}
  (\nabla u,\nabla v)_\triangle
  = \lambda (u,v)_\triangle
  \qquad \forall v\in H^1_0(\triangle).
\end{equation}

Since the inverse of the Laplacian is a compact and self-adjoint operator, the spectral theorem guarantees that \eqref{eq:eigenvalue-problem} admits a countably infinite sequence of eigenvalues, which can be arranged as
\[
  0 < \lambda_1(\triangle) < \lambda_2(\triangle) \le \lambda_3(\triangle) \le \cdots,
\]
each of finite multiplicity, and satisfying $\lim_{i\to \infty}\lambda_i(\triangle)=+\infty$.

Next, we specialize the general shape derivative formulas in \cite{Bogosel-Bucur-2024} to triangular domains and derive explicit formulas for the first- and second-order directional derivatives with respect to vertex perturbations.

\subsection{Derivative formulas}

For $p=(x_0,y_0)\in\mathbb{R}^2$, let $\triangle^p$ be the triangular domain with vertices $(0,0)$, $(1,0)$, and $(x_0,y_0)$. We denote by $\lambda_i^p$ the $i$th Dirichlet eigenvalue on $\triangle^p$, and by $u_i^p$ an associated $L^2$-normalized eigenfunction.

We first recall Hadamard's formula for the first-order directional derivative of a simple eigenvalue.

\begin{lemma}\label{lem:first-order-formula}
For $e=(a,b)\in\mathbb{R}^2$, assume that $\lambda_i^{p+te}$ remains simple for all $t\in(-\varepsilon,\varepsilon)$. Then the directional derivative of $\lambda_i^p$ in the direction $e$ is given by
\begin{equation}
    \lim_{t\to 0}\frac{\lambda_i^{p+te}-\lambda_i^p}{t}
    =
    \left(\dot P \nabla u_i^p,\nabla u_i^p\right)_{\triangle^p}
    \, (=: \dot\lambda_i^p),
\end{equation}
where the matrix $\dot P$ is defined by
\begin{equation}
    \dot P =
    \begin{pmatrix}
        0 & \displaystyle -\frac{a}{y_0} \\[10pt]
        \displaystyle -\frac{a}{y_0} & \displaystyle -\frac{2b}{y_0}
    \end{pmatrix}.
\end{equation}
\end{lemma}

\begin{proof}
     This follows by a formula given in  \cite[Theorem 2.1(i)]{Bogosel-Bucur-2024} (see also  \cite{Laurain-2020}) with the computation carried out for triangles. 
\end{proof}

Recall the second‐order derivative formula, which involves the material derivative of the eigenfunction.

\begin{lemma}\label{lem:second-order-difference-quotient-formula}
    For $e=(a,b)\in\mathbb{R}^2$, assume that $\lambda_i^{p+te}$ remains simple for all $t\in(-\varepsilon,\varepsilon)$. Then, the second‐order directional derivative of $\lambda_i^p$ in the direction $e$ is given by
    \begin{equation}\label{eq:second-order-difference-quotient-formula}
        \lim_{t\to 0}\frac{\lambda_i^{p+te}+\lambda_i^{p-te}-2\lambda_i^{p}}{t^2}=
        \left(\ddot P\nabla u_i^p,\nabla u_i^p\right)_{\triangle^p}-2\left(\dot P\nabla \dot u_i^p,\nabla u_i^p\right)_{\triangle^p}(=:\ddot\lambda_i^p),
    \end{equation}
    where the matrix $\ddot{P}$ is defined by
    \begin{equation*}
        \ddot{P} :=
        \begin{pmatrix}
        \displaystyle \frac{2a^2}{y_0^2} & \displaystyle \frac{4ab }{y_0^2} \\[12pt]
        \displaystyle \frac{4ab}{y_0^2} & \displaystyle \frac{6b^2}{y_0^2}
        \end{pmatrix}.
    \end{equation*}
    Here, $\dot u_i^p$ denotes the material derivative of the eigenfunction, i.e. the unique function $\dot u_i^p\in H^1_0(\triangle^p)$ satisfying the following variational equation: find $\dot u_i^p\in H^1_0(\triangle^p)$ such that
    \begin{equation}
    \begin{cases}
        (\nabla \dot u_i^p,\nabla v)_{\triangle^p}-\lambda_i^p(\dot u_i^p,v)_{\triangle^p}=-(\dot{P}\nabla u_i^p,\nabla v)_{\triangle^p}+\dot{\lambda}_i^p\cdot (u_i^p,v)_{\triangle^p}~~\forall v\in H^1_0(\triangle^p),\\
        (\dot u_i^p,u^p_i)_{\triangle^p}=0.
    \end{cases}
    \end{equation}
\end{lemma}

\begin{proof}
     This follows by a formula given in  \cite[Theorem 2.1(ii)]{Bogosel-Bucur-2024} (see also  \cite{Laurain-2020}) with the computation done on triangles. 
\end{proof}

\section{Main Theory}\label{eq:est-drivatives}

In this section, we derive a lower bound for the discrete approximation of the second-order directional derivative, together with an a priori error bound for the discrete first-order shape derivative.

\subsection{Lower bound estimate for the second-order directional derivative}

Recall that
\begin{equation*}
    \ddot{\lambda}_i^p
    = \bigl(\ddot P \nabla u_i^p, \nabla u_i^p\bigr)_{\triangle^p}
      - 2 \bigl(\dot P \nabla\dot u_i^p, \nabla u_i^p\bigr)_{\triangle^p},
\end{equation*}
where the material derivative $\dot u_i^p$ is characterized as the unique solution of the following variational equation.

Find $\dot u_i^p \in H_0^1(\triangle^p)$ such that
\begin{equation}\label{eq:problem-w}
\begin{cases}
    (\nabla \dot u_i^p, \nabla v)_{\triangle^p}
    - \lambda_i^p (\dot u_i^p,v)_{\triangle^p}
    = -(\dot P \nabla u_i^p, \nabla v)_{\triangle^p}
      + \dot{\lambda}_i^p (u_i^p, v)_{\triangle^p}
      & \forall v \in H_0^1(\triangle^p),\\[0.3em]
    (\dot u_i^p, u_i^p)_{\triangle^p} = 0.
\end{cases}
\end{equation}

The following lemma gives the spectral expansion of the material derivative $\dot u_i^p$.
\begin{lemma}[]
\label{lem:w-expansion-convergence}
Suppose that $\lambda_i^p$ is simple. Then, the material derivative $\dot u_i^p$ admits the expansion
\begin{equation}\label{eq:w-expansion-H1}
  \dot u_i^p=\sum_{\substack{k=1\\k\neq i}}^\infty c_k\,u_k^p
  \quad\text{ in }H_0^1(\triangle^p)
\end{equation}
where
\begin{equation}\label{eq:ck-formula-H1}
  c_k=-\frac{(\dot P\nabla u_i^p,\nabla u_k^p)_{\triangle^p}}{\lambda_k^p-\lambda_i^p}~~
  (k\neq i).
\end{equation}
\end{lemma}

\begin{proof}
Since the family \(\{u_k^p\}_{k\ge1}\) is a complete orthogonal basis of
\(H^1_0(\triangle^p)\), for \(\dot u_i^p\in H^1_0(\triangle^p)\), there exist coefficients
\(c_k\in\mathbb R\) such that
\[
  \dot u_i^p=\sum_{k=1}^{\infty}c_k u_k^p
  \qquad\text{in }H^1_0(\triangle^p).
\]
The orthogonality condition $(\dot{u}_i^p,u^p_i)_{\triangle^p}=0$ and the $L^2$–orthonormality of $\{u^p_k\}_{k\ge1}$ imply $c_i = 0$, therefore, the expansion reduces to \eqref{eq:w-expansion-H1}.

Next, take $v = u^p_j$ with $j \neq i$ in the first equation of \eqref{eq:problem-w}. Using
\[
    (\nabla u^p_k, \nabla u^p_j)_{\triangle^p}
    = \lambda_k^p (u^p_k, u^p_j)_{\triangle^p}
    = \lambda_k^p \delta_{kj},
\]
we obtain
\[
    (\nabla \dot{u}_i^p, \nabla u^p_j)_{\triangle^p}
    - \lambda_i^p (\dot{u}_i^p, u^p_j)_{\triangle^p}
    = c_j (\lambda_j^p - \lambda_i^p).
\]
On the other hand, we have
\[
    -(\dot P \nabla u^p_i, \nabla u^p_j)_{\triangle^p}
    + \dot{\lambda}_i^p (u^p_i, u^p_j)_{\triangle^p}
    = -(\dot P \nabla u^p_i, \nabla u^p_j)_{\triangle^p}
\]
for $j \neq i$. Equating both sides yields
\[
    c_j (\lambda_j^p - \lambda_i^p)
    = -(\dot P \nabla u^p_i, \nabla u^p_j)_{\triangle^p}.
\]
Since $\lambda_i^p$ is simple, we have $\lambda_j^p \neq \lambda_i^p$ for $j \neq i$, and hence
\[
    c_j
    = -\,\frac{(\dot P \nabla u^p_i, \nabla u^p_j)_{\triangle^p}}
              {\lambda_j^p - \lambda_i^p},
\]
which is precisely \eqref{eq:ck-formula-H1}.
\end{proof}

Using the above expansion for $\dot u_i^p$, we obtain the following formula for the second-order directional derivative $\ddot\lambda^p_i$:
\begin{lemma}[]
\label{lem:second_derivative_spectral}
Suppose that $\lambda_i^p$ is simple.
Then, its second-order directional derivative in the direction $e$ is given by
\begin{equation}\label{eq:ddlami-expansion}
    \ddot\lambda^p_i
    = \bigl(\ddot P\nabla u_i^p,\nabla u_i^p\bigr)_{\triangle^p}
      + 2 \sum_{\substack{k=1\\ k \neq i}}^{\infty}
        \frac{\bigl|\ip{\dot{P}\nabla u_i^p}{\nabla u_k^p}_{\triangle^p}\bigr|^2}{\lambda_k^p - \lambda_i^p}.
\end{equation}
\end{lemma}

\begin{proof}
From Lemma \ref{lem:second-order-difference-quotient-formula}, we have
\begin{equation}\label{eq:deruv-formula}
    \ddot\lambda^p_i
    = \bigl(\ddot P\nabla u_i^p,\nabla u_i^p\bigr)_{\triangle^p}
      - 2\bigl(\dot P\nabla \dot u_i^p,\nabla u_i^p\bigr)_{\triangle^p}.
\end{equation}
Since $\lambda_i^p$ is simple, by Lemma \ref{lem:w-expansion-convergence}, $\dot u_i^p$ admits the expansion
\[
    \dot u_i^p = \sum_{\substack{k=1\\ k \neq i}}^{\infty} c_k u_k^p,
    \qquad
    c_k = -\frac{\ip{\dot{P}\nabla u_i^p}{\nabla u_k^p}_{\triangle^p}}{\lambda_k^p - \lambda_i^p}.
\]
Hence, we have
\begin{align*}
    -2\bigl(\dot P\nabla \dot u_i^p,\nabla u_i^p\bigr)_{\triangle^p}
    &= -2\left(\dot P \sum_{\substack{k=1\\ k \neq i}}^{\infty} c_k \nabla u_k^p,
                \nabla u_i^p\right)_{\triangle^p} \\
    &= -2 \sum_{\substack{k=1\\ k \neq i}}^{\infty}
        c_k \ip{\dot P \nabla u_k^p}{\nabla u_i^p}_{\triangle^p} \\
    &= 2 \sum_{\substack{k=1\\ k \neq i}}^{\infty}
        \frac{\ip{\dot{P}\nabla u_i^p}{\nabla u_k^p}_{\triangle^p}\ip{\dot P \nabla u_k^p}{\nabla u_i^p}_{\triangle^p}}
             {\lambda_k^p - \lambda_i^p} \\
    &= 2 \sum_{\substack{k=1\\ k \neq i}}^{\infty}
        \frac{\bigl|\ip{\dot{P}\nabla u_i^p}{\nabla u_k^p}_{\triangle^p}\bigr|^2}{\lambda_k^p - \lambda_i^p}.
\end{align*}
Substituting this into \eqref{eq:deruv-formula} yields the claim.
\end{proof}

For $N\in\mathbb{N}$ with $i\leq N$, introduce the truncated quantity of \eqref{eq:ddlami-expansion} in Lemma \ref{lem:second_derivative_spectral}:
\begin{equation}\label{eq:lam-i-N}
\ddot\lambda^p_{i,N}
    := \bigl(\ddot P\nabla u_i^p,\nabla u_i^p\bigr)_{\triangle^p}
       + 2 \sum_{\substack{k=1 \\ k \neq i}}^{N}
         \frac{\bigl|\ip{\dot{P}\nabla u_i^p}{\nabla u_k^p}_{\triangle^p}\bigr|^2}{\lambda_k^p - \lambda_i^p}.    
\end{equation}
Note that $\ddot\lambda^p_{i,N}$ becomes a lower bound of $\ddot\lambda^p_{i}$:
\begin{equation}\label{eq:theoretical-lower}
    \ddot\lambda^p_{i}\geq\ddot\lambda^p_{i,N}.
\end{equation}

In the following subsection, we derive a computable lower bound for $\ddot\lambda^p_{i}$.

\subsection{Finite element approximation of eigenpairs}

Let $\mathcal T^h$ be a regular triangulation of $\triangle^p$ with mesh size
$h$, i.e. $h$ is the maximum edge length in $\mathcal T^h$.
For each element $K\in\mathcal{T}^h$, denote by $P^1(K)$ the space of polynomials
of degree at most~$1$.

The associated Lagrange finite element space is
\[
L_h(\triangle^p)
:= \bigl\{\hat v\in C^0(\triangle^p)\;\big|\;\hat v|_K\in P^1(K)\text{ for all }K\in\mathcal{T}^h\bigr\},
\]
and we set
\[
V_h(\triangle^p):=H^1_0(\triangle^p)\cap L_h(\triangle^p).
\]
We also write \(V_h^{\CG}(\triangle^p):=V_h(\triangle^p)\).

To approximate the eigen-pairs, we consider the following
discrete eigenvalue problem: Find $\hat u^p\in V_h(\triangle^p)\setminus\{0\}$ and $\hat \lambda^p>0$ such that
\begin{equation}\label{eq:cg-problem}
  (\nabla \hat u^p,\nabla \hat v)_{\triangle^p}
  =\hat \lambda^p\, (\hat u^p,\hat v)_{\triangle^p}
  \qquad \forall \hat v\in V_h(\triangle^p).
\end{equation}

Let $N_{\dim}=\dim V_h(\triangle^p)$.  
The eigenvalues are ordered as
\[
0<\hat\lambda_1^p\le\hat\lambda_2^p\le\cdots\le\hat\lambda_{N_{\dim}}^p.
\]
Set \(\lambda_{k,h}^{\CG}(\triangle^p):=\hat\lambda_k^p\).

To obtain lower bounds of eigenvalues, introduce the Crouzeix--Raviart finite element space over \(\mathcal T^h\):
\begin{equation}
\label{def:fem-space-cr}
V_h^{\CR}(\triangle^p)
:=
\left\{v_h \in L^2(\triangle^p) :
\begin{aligned}
&v_h|_K \in P^1(K) \quad \forall K \in \mathcal T^h,\\
&v_h \text{ is continuous at the midpoints of internal edges},\\
&v_h \text{ vanishes at the midpoints of boundary edges}
\end{aligned}
\right\}.
\end{equation}
Let \(\lambda_{k,h}^{\CR}(\triangle^p)\) be the \(k\)-th eigenvalue of the following CR finite element eigenvalue problem: Find \(u_h\in V_h^{\CR}(\triangle^p)\setminus\{0\}\) and \(\lambda_h>0\) such that
\begin{equation}\label{eq:cr-problem}
  (\nabla_h u_h,\nabla_h v_h)_{\triangle^p}
  =\lambda_h (u_h,v_h)_{\triangle^p}
  \qquad \forall v_h\in V_h^{\CR}(\triangle^p),
\end{equation}
where \(\nabla_h\) denotes the elementwise gradient. Let \(N_{\CR}=\dim V_h^{\CR}(\triangle^p)\).



By using the conforming and CR finite element approximations, we obtain the following rigorous eigenvalue bounds.
\begin{lemma}[Theorem 2.1 of \cite{liu2015framework}]\label{lem:est-tau}
    Let \(C_h=0.1893h\), where \(h\) is the maximum edge length of \(\mathcal T^h\). Then, for
    \(k=1,2,\dots,\min\{N_{\dim},N_{\CR}\}\), the following two-sided estimate holds:
    \begin{equation}\label{lem:l-estimation-cr}
      \underline\lambda_k
      :=\frac{\lambda_{k,h}^{\CR}(\triangle^p)}{1+C_h^2\lambda_{k,h}^{\CR}(\triangle^p)}
      \;\leq\;
      \lambda_k(\triangle^p)
      \;\leq\;
      \lambda_{k,h}^{\CG}(\triangle^p)
      =:\overline\lambda_k.
    \end{equation}
\end{lemma}

Although Lemma \ref{lem:est-tau} provides rigorous lower bounds for eigenvalues, it is not well suited to exploiting high-order finite element approximations. In practice, improving the accuracy of the resulting lower bounds then relies mainly on mesh refinement.

To compute highly accurate eigenvalue bounds more efficiently, we employ the approach of \cite[Chap.~5]{liu2024lehmann}, which formulates the Lehmann--Goerisch method \cite{goerisch1990determination,behnke1994inclusions,lehmann1963optimale} in a finite element setting. We emphasize that this method also depends on projection-based eigenvalue bounds \cite{liu2015framework}, as in Lemma \ref{lem:est-tau}.%
The version of the Lehmann--Goerisch theorem used in the present paper is summarized in Lemma \ref{thm:Lehmann--Goerisch}; for the full argument, see \cite[\S 5.2.1]{liu2024lehmann}.

\begin{lemma}\label{thm:Lehmann--Goerisch}
Let $\hat{u}_i \in V_h^{\CG}$ $(i=1,\ldots,n)$ be approximate eigenfunctions, and let $\hat w_i \in H(\operatorname{div},\triangle)$ satisfy
\begin{equation}
\label{eq:w_term_in_LH_method}
(\hat w_i,\nabla v) = (\hat u_i,v)\quad \forall v \in H^1_0(\triangle).
\end{equation}
Let $\lambda_{n,h}$ denote the $n$th eigenvalue obtained from the variational eigenvalue problem in $V_h^{\CG}$. 
Assume that $\rho$ is a lower bound for the $(n+1)$th exact eigenvalue, namely $\rho \le \lambda_{n+1}$, and that $\rho > \lambda_{n,h}$.
Define the $n\times n$ matrices
\begin{equation}
A_0 := \bigl((\nabla \hat u_i,\nabla \hat u_j)_{T^0}\bigr)_{i,j=1}^n,\quad
A_1 := \bigl((\hat u_i,\hat u_j)_{T^0}\bigr)_{i,j=1}^n,\quad
A_2 := \bigl((\hat w_i,\hat w_j)_{T^0}\bigr)_{i,j=1}^n,
\end{equation}
and
\begin{equation}
A := A_0-\rho A_1,\qquad
B := A_0-2\rho A_1+\rho^2 A_2.
\end{equation}
Suppose further that $B$ is positive definite.
Let $\tau_1\le \tau_2\le \cdots \le \tau_n$ be the eigenvalues of the generalized eigenvalue problem
\[
Az=\mu Bz.
\]
Then the following lower bounds hold:
\begin{equation}
\lambda_k \ge \rho-\frac{\rho}{1-\tau_{n+1-k}}\qquad (1\le k\le n).
\end{equation}
\end{lemma}

For the construction of $\widehat w_i$, a practical choice is to approximate it in the Raviart--Thomas finite element space $\mathrm{RT}_h$, which serves as a discrete counterpart of $H(\operatorname{div})$. The Raviart--Thomas space of degree $d$ is defined by
\begin{equation}
\mathrm{RT}_h^d
:=
\left\{
p_h\in H(\operatorname{div})
\;\middle|\;
p_h=(a_K,b_K)+c_K\mathbf{x},\ 
a_K,b_K,c_K\in P^d(K)
\text{ for each } K\in\mathcal{T}^h
\right\}.
\end{equation}
We also introduce the discontinuous polynomial space
\begin{equation}
X_h^d
:=
\left\{
v_h
\;\middle|\;
v_h\in P^d(K)
\text{ for each } K\in\mathcal{T}^h
\right\},
\end{equation}
which is used in the computation of $\widehat w_i$.

Since the solution of \eqref{eq:w_term_in_LH_method} is not unique, we determine an optimal choice of $\widehat w_i$ through the following mixed formulation: find $(\hat p_i,\hat g_i)\in \mathrm{RT}_h^d\times X_h^d$ such that
\begin{equation}
\begin{cases}
(\hat p_i,\hat q)+(\hat g_i,\operatorname{div}\hat q)=0
\quad \forall \hat q\in \mathrm{RT}_h^d,\\
(\operatorname{div}\hat p_i,\hat f)+(\hat u_i,\hat f)=0
\quad \forall \hat f\in X_h^d.
\end{cases}
\end{equation}
In our computation of high-precision eigenvalue bounds, we choose $d=2$ or $3$.

To obtain the concrete lower bound of \eqref{eq:lam-i-N}, define the corresponding computable $N$-term approximation by
\begin{equation}
    \widehat{(\ddot\lambda^p_{i,N})}
    := \bigl(\ddot P\nabla \hat u_i^p,\nabla \hat u_i^p\bigr)_{\triangle^p}
       + 2 \sum_{\substack{k=1 \\ k \neq i}}^{N}
         \frac{\bigl|\ip{\dot{P}\nabla \hat u_i^p}{\nabla \hat u_k^p}_{\triangle^p}\bigr|^2}{\lambda_k^p - \lambda_i^p}.
\end{equation}
Also, let $\mathrm{Est}_{a}(u_k^p ,\hat{u}_k^p)$ $(k=1,\cdots,N)$ be the quantity satisfying
\begin{equation}\label{eq:def-of-estimators}
    \|\nabla (u_k^p - \hat{u}_k^p)\|_{\triangle^p} \le \mathrm{Est}_{a}(u_k^p ,\hat{u}_k^p).
\end{equation}
Note that, such a constant $\mathrm{Est}_{a}(u_k^p ,\hat{u}_k^p)$ can be obtained using Lemma \ref{lem:eigenvec-estimation-algorithm-1} if $\lambda_k$ is well separated from $\lambda_{k-1}$ and $\lambda_{k+1}$.

\medskip

As a preparation, we first establish an a priori error bound for the discrete first-order shape derivative, defined by
\begin{equation*}
\widehat{(\dot\lambda_i^p)}:=
\left(\dot P\nabla \hat{u}_i^p,\nabla \hat{u}_i^p\right)_{\triangle^p}.
\end{equation*}

\begin{lemma}[A priori error estimate for the discrete first-order shape derivative]\label{lem:first-order-error-estimate}
We have
\[
  \bigl|\dot\lambda_i^p-\widehat{(\dot\lambda_i^p)}\bigr|
  \;\le\;
  \|\dot P\|_{2}\,
  \Bigl(\sqrt{\lambda_i^p}+\sqrt{\hat\lambda_i^p}\Bigr)\,
  \mathrm{Est}_a(u_i^p,\hat{u}_i^p),
\]
where $\|\dot P\|_{2}$ denotes the spectral norm of the matrix~$\dot P$.
\end{lemma}

\begin{proof}
We write
\[
  \dot\lambda_i^p-\widehat{(\dot\lambda_i^p)}
  =\bigl(\dot P\nabla u_i^p,\nabla u_i^p\bigr)_{\triangle^p}
   -\bigl(\dot P\nabla\hat u_i^p,\nabla\hat u_i^p\bigr)_{\triangle^p}.
\]
Adding and subtracting the mixed term gives
\[
\begin{aligned}
  \dot\lambda_i^p-\widehat{(\dot\lambda_i^p)}
  &=
     \bigl(\dot P(\nabla u_i^p-\nabla\hat u_i^p),\nabla u_i^p\bigr)_{\triangle^p}
     +\bigl(\dot P\nabla\hat u_i^p,\nabla u_i^p-\nabla\hat u_i^p\bigr)_{\triangle^p}.
\end{aligned}
\]
By the Cauchy--Schwarz inequality, we have
\[
\begin{aligned}
  \bigl|\dot\lambda_i^p-\widehat{(\dot\lambda_i^p)}\bigr|
  &\le
     \|\dot P\|_{2}\,
     \|\nabla u_i^p-\nabla\hat u_i^p\|_{\triangle^p}\,
     \Bigl(\|\nabla u_i^p\|_{\triangle^p}
           +\|\nabla\hat u_i^p\|_{\triangle^p}\Bigr).
\end{aligned}
\]
Since the eigenfunctions are $L^{2}(\triangle^p)$-normalised, we have
$\|\nabla u_i^p\|_{\triangle^p}^{2}=\lambda_i^p$ and
$\|\nabla\hat u_i^p\|_{\triangle^p}^{2}=\hat\lambda_i^p$.
Substituting these identities yields the stated estimate.
\end{proof}

\medskip

For \eqref{eq:lam-i-N}, we obtain the following computable lower bound.
\begin{theorem}
\label{thm:main-theorem-est}
Suppose $i\leq N$ and  that $\lambda_i^p$ is simple.
Then, the second-order directional derivative $\ddot\lambda^p_{i}$ satisfies
\begin{align}\label{eq:main-theorem-est}
    \ddot\lambda^p_{i}
    \ge \widehat{(\ddot\lambda_{i,N}^p)}
           -\left\|\ddot P\right\|_{2}
            \bigl(\sqrt{\lambda_i^p}+\sqrt{\hat\lambda_i^p}\bigr)~
            \mathrm{Est}_{a}(u_i^p,\hat{u}_i^p)- 2 \sum_{\substack{k=1 \\ k \neq i}}^{N}
            \frac{\mathcal{A}_{ik}}{|\lambda_k^p - \lambda_i^p|}
            \qquad (=:~\underline{\ddot\lambda}^p_{i}), 
\end{align}
where
\begin{align*}
    \mathcal{A}_{ik}
    &:= \|\dot P\|_{2}^2
        \Bigl(\mathrm{Est}_{a}(u_i^p,\hat{u}_i^p)\sqrt{\hat{\lambda}_k^p}
              + \sqrt{\lambda_i^p}~\mathrm{Est}_{a}(u_k^p,\hat{u}_k^p)\Bigr)
        \Bigl(\sqrt{\lambda_i^p\lambda_k^p}
              + \sqrt{\hat\lambda_i^p\hat\lambda_k^p}\Bigr).
\end{align*}
\end{theorem}

\begin{proof}
Let us first estimate the difference between $\ddot\lambda^p_{i,N}$ and $\widehat{(\ddot\lambda_{i,N}^p)}$.
By the definitions of $\ddot\lambda^p_{i,N}$ and $\widehat{(\ddot\lambda_{i,N}^p)}$,
\begin{align}
\label{eq:main-est-error}
    &|\ddot\lambda^p_{i,N}-\widehat{(\ddot\lambda_{i,N}^p)}|\nonumber\\
    &=
    \left|
      (\ddot P\nabla u_i^p,\nabla u_i^p)_{\triangle^p}
      -(\ddot P\nabla \hat{u}_i^p,\nabla \hat{u}_i^p)_{\triangle^p} 
      + 2\sum_{\substack{k=1 \\ k \neq i}}^{N}
        \frac{\bigl|\ip{\dot{P}\nabla u_i^p}{\nabla u_k^p}_{\triangle^p}\bigr|^2
             -\bigl|\ip{\dot{P}\nabla\hat  u_i^p}{\nabla \hat u_k^p}_{\triangle^p}\bigr|^2}
             {\lambda_k^p - \lambda_i^p}
    \right|\nonumber \\
    &\le
    \|\ddot P\|_{2}
      \bigl(\sqrt{\lambda_i^p}+\sqrt{\hat\lambda_i^p}\bigr)
      ~\mathrm{Est}_{a}(u_i^p,\hat{u}_i^p)\nonumber \\
    &~
      + 2\sum_{\substack{k=1 \\ k \neq i}}^{N}
        \frac{1}{|\lambda_k^p - \lambda_i^p|}
        \bigl|
          \ip{\dot{P}\nabla \hat{u}_i^p}{\nabla \hat{u}_k^p}_{\triangle^p}
          -\ip{\dot{P}\nabla u_i^p}{\nabla u_k^p}_{\triangle^p}
        \bigr|
        \bigl|
          \ip{\dot{P}\nabla \hat{u}_i^p}{\nabla \hat{u}_k^p}_{\triangle^p}
          +\ip{\dot{P}\nabla u_i^p}{\nabla u_k^p}_{\triangle^p}
        \bigr|.
\end{align}
For $k=1,\dots,N$, we have
\begin{align}
\label{eq:main-est-each}
    \bigl|
      \ip{\dot{P}\nabla \hat{u}_i^p}{\nabla \hat{u}_k^p}_{\triangle^p}
      -\ip{\dot{P}\nabla u_i^p}{\nabla u_k^p}_{\triangle^p}
    \bigr|\nonumber
    &\le
    \bigl|
      \ip{\dot{P}\nabla \hat{u}_i^p}{\nabla \hat{u}_k^p}_{\triangle^p}
      -\ip{\dot{P}\nabla u_i^p}{\nabla \hat{u}_k^p}_{\triangle^p}
    \bigr|\nonumber\\
    &~~~+
    \bigl|
      \ip{\dot{P}\nabla u_i^p}{\nabla \hat{u}_k^p}_{\triangle^p}
      -\ip{\dot{P}\nabla u_i^p}{\nabla u_k^p}_{\triangle^p}
    \bigr|\nonumber \\
    &\le
    \|\dot{P}\|_2
    \Bigl(
      \mathrm{Est}_{a}(u_i^p,\hat{u}_i^p)\sqrt{\hat{\lambda}_k^p}
      + \sqrt{\lambda_i^p}~\mathrm{Est}_{a}(u_k^p,\hat{u}_k^p)
    \Bigr).
\end{align}
Also, it follows that
\begin{align}
\begin{split}\label{eq:main-est-each-2}
    \bigl|
      \ip{\dot{P}\nabla \hat{u}_i^p}{\nabla \hat{u}_k^p}
      +\ip{\dot{P}\nabla u_i^p}{\nabla u_k^p}
    \bigr|
    &\le
    \|\dot P\|_{2}
    \left(\sqrt{\hat\lambda_i^p\hat\lambda_k^p}
          +\sqrt{\lambda_i^p\lambda_k^p}\right).
\end{split}
\end{align}
Combining \eqref{eq:main-est-error}, \eqref{eq:main-est-each} and \eqref{eq:main-est-each-2}, we obtain
\begin{equation}\label{eq:almost-final-est}
    |\ddot\lambda^p_{i,N}-\widehat{(\ddot\lambda_{i,N}^p)}|
    \le
    \|\ddot P\|_{2}
      \bigl(\sqrt{\lambda_i^p}+\sqrt{\hat\lambda_i^p}\bigr)
      \mathrm{Est}_{a}(u_i^p,\hat{u}_i^p)
      + 2 \sum_{\substack{k=1 \\ k \neq i}}^{N}
        \frac{\mathcal{A}_{ik}}{|\lambda_k^p - \lambda_i^p|}.
\end{equation}
Finally, from \eqref{eq:almost-final-est} and \eqref{eq:theoretical-lower}, it follows that
\[
    \ddot\lambda^p_{i}
    \ge \widehat{(\ddot\lambda_{i,N}^p)}
        -\|\ddot P\|_{2}
         \bigl(\sqrt{\lambda_i^p}+\sqrt{\hat\lambda_i^p}\bigr)
         \mathrm{Est}_{a}(u_i^p,\hat{u}_i^p)
        - 2 \sum_{\substack{k=1 \\ k \neq i}}^{N}
          \frac{\mathcal{A}_{ik}}{|\lambda_k^p - \lambda_i^p|}.
\]
\end{proof}

\section{Computer-assisted proof of Theorem \ref{thm:main-problem}}\label{Sec:Proof-of-Lower}

Recall, to prove Conjectures \ref{Conj:1}, \ref{Conj:2} and \ref{Conj:3}, it suffices to establish the following Theorem \ref{thm:main-problem}. We outline the strategy of the computer-assisted proof of
Theorem~\ref{thm:main-problem}.

Let $\triangle^p$ denote the triangle with vertices $(0,0)$, $(1,0)$, and
$p=(x,y)$. We normalize triangles by requiring that
\[
\Omega:=\bigl\{(x,y)\in\mathbb{R}^2: x^2+y^2\le 1,\ x\ge\tfrac12,\ y>0\bigr\}
\]
is the moduli set for the third vertex. Let
\[
p_0:=\Bigl(\tfrac12,\tfrac{\sqrt3}{2}\Bigr).
\]

We divide $\Omega$ into the following three subregions, whose union covers
$\Omega$:
\begin{align}
\Omega_{\mathrm{up}}
&:= \left\{(x,y)\in\Omega : y\ge \sqrt{3}/2-\varepsilon_{\mathrm{up}}\right\},
\label{def:sub-domain-up}\\
\Omega_{\mathrm{mid}}
&:= \left\{(x,y)\in\Omega :
\varepsilon_{\mathrm{down}}\le y,\ 
(x-1/2)^2+(y-\sqrt{3}/2)^2\ge \varepsilon_{\mathrm{up}}^2\right\},
\label{def:sub-domain-mid}\\
\Omega_{\mathrm{down}}
&:= \left\{(x,y)\in\Omega : y<\varepsilon_{\mathrm{down}}\right\},
\label{def:sub-domain-down}
\end{align}
where
\begin{equation}\label{eq:epsilon_up_down}
   \varepsilon_{\mathrm{up}}=0.122,
\qquad
\varepsilon_{\mathrm{down}}=0.04. 
\end{equation}

For the illustration of this geometric setting, see Figure \ref{fig:omega}.

The proof proceeds by verifying the positivity of $J_k$ $(k=1,2)$ on each of
these regions.

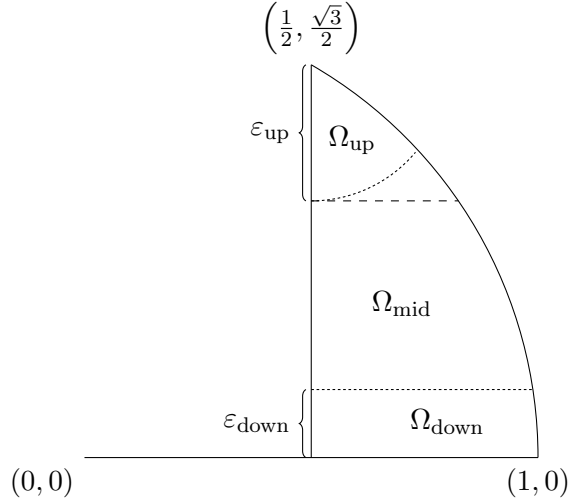
\begin{figure}[H]
    \centering
    \begin{tikzpicture}[scale=6,>=stealth]

  \def\eps{0.3} 

  \coordinate (O) at (0,0);
  \coordinate (A) at (0,0);                     
  \coordinate (B) at (1,0);                     
  \coordinate (C) at ({1/2},{sqrt(3)/2});       
  \coordinate (X) at ({1/2},0);                 
  \coordinate (P) at ($(C)+(0,-\eps)$);         

  \coordinate (R) at ({atan(sqrt(3) - 2*\eps)}:1);

  \draw (A) -- (B);
  \draw[name path=outer] (B) arc[start angle=0,end angle=60,radius=1];
  \draw (X) -- (C);


  
\draw[dash pattern=on 1pt off 1pt] (0.5,0.15) -- ({sqrt(1-0.15*0.15)},0.15);

  \draw[dashed]
  (P) -- ({sqrt(1-(sqrt(3)/2-\eps)*(sqrt(3)/2-\eps))},
          {sqrt(3)/2-\eps});

  
\draw[dash pattern=on 1pt off 1pt, name path=inner]
    (C) ++(-90:\eps) arc[start angle=-90,end angle=-40,radius=\eps];

  \draw[decorate,decoration={brace,mirror,raise=2pt}]
    (C) -- (P) node[midway,left=4pt] {$\varepsilon_{\mathrm{up}}$};

\draw[decorate,decoration={brace,raise=2pt}]
    (X) -- (0.5,0.15) node[midway,left=4pt] {$\varepsilon_{\mathrm{down}}$};

  \node at (0.59,0.70) {$\Omega_{\text{up}}$};
  \node at (0.70,0.34) {$\Omega_{\text{mid}}$};
  \node at (0.80,0.08) {$\Omega_{\text{down}}$};

  \node[below left] at (A) {$(0,0)$};
  \node[below]      at (B) {$(1,0)$};
  \node[above]      at (C) {$\left(\tfrac12,\tfrac{\sqrt3}{2}\right)$};

\end{tikzpicture}
    \caption{Subregions $\Omega_{\mathrm{up}}$, $\Omega_{\mathrm{mid}}$ and $\Omega_{\mathrm{down}}$}
    \label{fig:omega}
\end{figure}

\medskip

\textbf{Step 1 (Region $\Omega_{\mathrm{up}}$).}
\begin{enumerate}\itemsep3pt
\item[\emph{(1-1)}]
Using the first-order shape derivative formula, we prove analytically that
\[
\nabla J_k(\triangle^{p_0})=0
\qquad (k=1,2),
\]
that is, $p_0$ is a stationary point of $J_k$ for $k=1,2$.

\item[\emph{(1-2)}]
We compute a rigorous lower bound for the second-order derivative
$\partial_x^2 J_k$ $(k=1,2)$ and prove its positivity. The resulting strict
convexity in the $x$-direction implies that any minimizer in
$\Omega_{\mathrm{up}}$ must lie on the symmetry axis $x=1/2$.

\item[\emph{(1-3)}]
We then evaluate $\partial_y^2 J_k$ $(k=1,2)$ on the symmetry axis $x=1/2$
and verify that it is strictly positive.

\item[]
Combining these results, we conclude that the equilateral triangle
$\triangle^{p_0}$ is the unique minimizer of $J_k$ $(k=1,2)$ in
$\Omega_{\mathrm{up}}$.
\end{enumerate}

\medskip

\textbf{Step 2 (Region $\Omega_{\mathrm{mid}}$).}
\begin{enumerate}\itemsep3pt
\item[]
We compute rigorous lower bounds $\underline{J_k}(\triangle^p)$ $(k=1,2)$ such
that
\[
(J_k(\triangle^{p_0})=)\ 0 < \underline{J_k}(\triangle^p)\le J_k(\triangle^p)
\qquad\text{for all }p\in\Omega_{\mathrm{mid}}.
\]
This excludes the possibility that a global minimizer lies in
$\Omega_{\mathrm{mid}}$.
\end{enumerate}

\medskip

\textbf{Step 3 (Region $\Omega_{\mathrm{down}}$).}
\begin{enumerate}\itemsep3pt
\item[]
We show analytically that a minimizer cannot lie in the nearly degenerate
region $\Omega_{\mathrm{down}}$.
\end{enumerate}

\begin{remark}\label{rem:J1J2not-ordered}
There is no global ordering between $J_1$ and $J_2$ on the class of triangles.
Indeed, for
\[
\triangle=\operatorname{conv}\{(0,0),(1,0),(s,t)\},
\]
we obtain the following certified inequalities:
\[
J_1(\triangle)-J_2(\triangle) > 7.5882748973\times 10^{-8}
\quad\text{for }(s,t)=(0.5,\,0.8560254038),
\]
whereas
\[
J_1(\triangle)-J_2(\triangle) < -5.3967328847
\quad\text{for }(s,t)=(0.5,\,0.1).
\]
Hence one needs to treat the proofs of each inequality separately. 
\end{remark}

\subsection{Auxiliary estimates}

In this subsection, we collect several estimates used in the
computer-assisted proof.

\medskip

\textbf{Perturbation of eigenvalues.}

\begin{lemma}[Appendix Lemma A.2 in \cite{endo2023shape}]
\label{lem:eig-perturbation-first}
Let $\lambda_i^p$ and $\lambda_i^{\tilde p}$ denote the $i$-th eigenvalues
$(i=1,2,\dots)$ on $\triangle^p$ and $\triangle^{\tilde p}$, respectively.
Then
\begin{equation}\label{eq:perturbation-first}
\lambda_{\min}\!\left(S_{p,\tilde p}^{-1}S_{p,\tilde p}^{-\intercal}\right)\lambda_i^p
\;\le\;
\lambda_i^{\tilde p}
\;\le\;
\lambda_{\max}\!\left(S_{p,\tilde p}^{-1}S_{p,\tilde p}^{-\intercal}\right)\lambda_i^p,
\end{equation}
where $S_{p,\tilde p}:\mathbb{R}^2\to\mathbb{R}^2$ denotes the linear map from
$\triangle^p$ onto $\triangle^{\tilde p}$.
\end{lemma}

\textbf{Error estimates for eigenspaces}

We next introduce distances between subspaces of $H^1_0(\triangle)$, which will be used
to quantify errors in numerical approximations of eigenspaces. For two subspaces
$E,E^h\subset H^1_0(\triangle)$, we define the directed distances
$\delta_a,\delta_b,\bar{\delta}_a,\bar{\delta}_b$ by
\begin{gather*}
  \delta_a(E,E^h)
  :=\max_{\substack{v\in E\\\|\nabla v\|_{\triangle}=1}}
      \min_{v_h\in E^h}\|\nabla v-\nabla v_h\|_{\triangle},
  \qquad
  \delta_b(E,E^h)
  :=\max_{\substack{v\in E\\\|v\|_{\triangle}=1}}
      \min_{v_h\in E^h}\|v-v_h\|_{\triangle},\\[0.3em]
  \bar{\delta}_a(E,E^h)
  :=\max_{\substack{v\in E\\\|v\|_{\triangle}=1}}
      \min_{\substack{v_h\in E^h\\\|v_h\|_{\triangle}=1}}
      \|\nabla v-\nabla v_h\|_{\triangle},
  \qquad
  \bar{\delta}_b(E,E^h)
  :=\max_{\substack{v\in E\\\|v\|_{\triangle}=1}}
      \min_{\substack{v_h\in E^h\\\|v_h\|_{\triangle}=1}}
      \|v-v_h\|_{\triangle}.
\end{gather*}

To describe errors for clusters of eigenfunctions, we introduce notation for
clusters of eigenvalues. Let $n_k$ and $N_k$ denote the indices of the first and
last eigenvalues in the $k$-th cluster; see Figure~\ref{fig:eigenvalue-distribution}.
Eigenvalues within a cluster are not assumed to be equal. For the fixed cluster $k$
under consideration, we set $n=n_k$ and $N=N_k$ to simplify the notation.

\begin{figure}[H]
  \centering
  \includegraphics[keepaspectratio, scale=0.18]{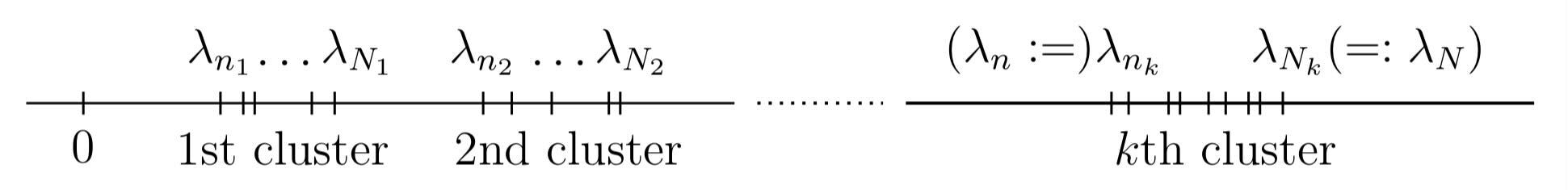}
  \caption{\label{fig:eigenvalue-distribution} Clusters of eigenvalues}
\end{figure}

Let $E_k$ be the space spanned by the exact eigenfunctions associated with the
$k$-th cluster:
\begin{equation}\label{eq:def-Ek}
  E_k=\operatorname{span}\{u_{n},u_{n+1},\dots,u_{N}\}.
\end{equation}
Similarly, let $u_{i,h}\in H^1_0(\triangle)$ be approximations of the exact
eigenfunctions $u_i$ for $i=n,\dots,N$, and define the corresponding approximate
eigenspace by
\begin{equation}\label{eq:def-Ehatk}
  E^h_{k}=\operatorname{span}\{u_{n,h},u_{n+1,h},\dots,u_{N,h}\}.
\end{equation}
We also define
\[
  \lambda_{n,h}:=\min_{u_h\in E^h_k}
     \frac{\|\nabla u_h\|^2_{\triangle}}{\|u_h\|_{\triangle}^2},
  \qquad
  \lambda_{N,h}:=\max_{u_h\in E^h_k}
     \frac{\|\nabla u_h\|^2_{\triangle}}{\|u_h\|_{\triangle}^2}.
\]

We define measures of non-orthogonality between $E$ and $E^h$ by
\begin{equation}
  \label{def:espilon_a_b}
  \varepsilon^h_a(E,E^h)
  =\max_{\substack{v\in E\\\|\nabla v\|_{\triangle}=1}}
    \max_{\substack{v_h\in E^h\\ \|\nabla v_h\|_{\triangle}=1}}
    (\nabla v,\nabla v_h)_{\triangle},
  \qquad
  \varepsilon^h_b(E,E^h)
  =\max_{\substack{v\in E\\\|v\|_{\triangle}=1}}
    \max_{\substack{v_h\in E^h\\ \|v_h\|_{\triangle}=1}}
    (v,v_h)_{\triangle}.
\end{equation}

We recall the following basic identity, which is valid for each exact eigenpair
$(\lambda_i,u_i)$ and each approximation $(\lambda_h,u_h)$ with
$u_h\in H^1_0(\triangle)$ and
$\|\nabla u_h\|^2_{\triangle}=\lambda_h\|u_h\|^2_{\triangle}$
(see, e.g., \cite[p.~55]{boffi2010finite}):
\begin{equation}\label{eq:fundamental-formula}
  \|\nabla u_i-\nabla u_h\|_{\triangle}^2
  =\lambda_i\|u_i-u_h\|_{\triangle}^2
   -(\lambda_i-\lambda_h)\|u_h\|_{\triangle}^2.
\end{equation}

The following estimate will be used to measure the error between clusters of
eigenfunctions. An important feature of this estimate is that it bounds the error
of the approximate eigenspace using only eigenvalue information and the mutual
orthogonality of the approximate eigenfunctions. Note that all eigenvalues
appearing in this paper can be bounded by Lemmas~\ref{lem:est-tau} and
\ref{thm:Lehmann--Goerisch}, or by the perturbation estimate of
Lemma~\ref{lem:eig-perturbation-first}.

\begin{lemma}[Theorem~1 of Liu--Vejchodsk{\`y} \cite{liu2022fully}]
  \label{lem:eigenvec-estimation-algorithm-1}
  Let $\rho>0$ satisfy $\lambda_{n_k}<\rho\leq\lambda_{N_k+1}$. Then
  \begin{align}\label{eq:estimation-for-delta}
    \begin{split}
      \delta^2_a(E_k,E^h_{k})
      &\leq
      \frac{\rho(\lambda_{N_k,h}-\lambda_{n_k})
            +\lambda_{n_k}\lambda_{N_k,h}\,\theta_a^{(k)}}
           {\lambda_{N_k,h}(\rho-\lambda_{n_k})}
      \;(:=\mathcal{E}^2_a(E_k,E^h_k)),\\
      \delta^2_b(E_k,E^h_{k})
      &\leq
      \frac{\lambda_{N_k,h}-\lambda_{n_k}+\theta_b^{(k)}}
           {\rho-\lambda_{n_k}}
      \;(:=\mathcal{E}^2_b(E_k,E^h_k)).
    \end{split}
  \end{align}
  Here
  \begin{align*}
    \theta_a^{(k)}
    &=\sum_{l=1}^{k-1}
      \frac{\rho-\lambda_{n_l}}{\lambda_{n_l}}
      \bigl[\varepsilon^h_a(E^h_{l},E^h_{k})
           +\delta_a(E_{l},E^h_{l})\bigr]^2,\\
    \theta_b^{(k)}
    &=\sum_{l=1}^{k-1}
      (\rho-\lambda_{n_l})
      \bigl[\varepsilon^h_b(E^h_{l},E^h_{k})
           +\delta_b(E_{l},E^h_{l})\bigr]^2.
  \end{align*}
\end{lemma}

\begin{lemma}[Lemma~2 of \cite{liu2022fully}]
  \label{lem:est-eta-b-h}
  Let $v_1,\dots,v_m$ and $v'_1,\dots,v'_{m'}$ be bases of subspaces $E$ and
  $E'$ of $H^1_0(\triangle)$, respectively. Define the matrices $F,G,H$ by
  \[
    F=\bigl((v_i,v'_j)_{\triangle}\bigr)_{m\times m'},\qquad
    G=\bigl((v_i,v_j)_{\triangle}\bigr)_{m\times m},\qquad
    H=\bigl((v'_i,v'_j)_{\triangle}\bigr)_{m'\times m'}.
  \]
  Then
  \[
    (\varepsilon^h_b(E,E'))^2
    =\lambda_{\max}\!\left(F^{\mathsf T}G^{-1}F,\,H\right)
    =\lambda_{\max}\!\left(FH^{-1}F^{\mathsf T},\,G\right),
  \]
  where $\lambda_{\max}(A,B)$ denotes the largest generalized eigenvalue of
  $Ax=\lambda Bx$. Moreover, suppose that
  \[
    \|F^{\mathsf T}F\|_2\le \eta_F,\qquad
    \|I-G\|_2\le \eta_G,\qquad
    \|I-H\|_2\le \eta_H.
  \]
  If $\eta_G,\eta_H<1$, then
  \[
    (\varepsilon^h_b(E,E'))^2
    \le \frac{\eta_F}{(1-\eta_G)(1-\eta_H)}.
  \]
\end{lemma}

\begin{lemma}[Lemma~3 of \cite{liu2022fully}]
  \label{lem:bbar-b-relation}
  For the distances $\bar{\delta}_a$ and $\delta_b$, the following relation holds:
  \begin{equation}\label{eq:bbar-b-relation}
    \overline{\delta}^2_a(E_k,E^h_k)
    \leq \lambda_{N_k}+\hat\lambda_{N_k}
          -2\lambda_{n_k}\sqrt{1-\delta^2_b(E_k,E^h_k)}.
  \end{equation}
\end{lemma}

Recall that $\mathrm{Est}_{a}(u_j,\hat{u}_j)$ denotes a constant satisfying
\begin{equation}\label{eq:def-of-estimators-c}
  \|\nabla (u_j-\hat{u}_j)\|
  \le \mathrm{Est}_{a}(u_j,\hat{u}_j).
\end{equation}

The algorithm below describes how to compute the estimator  $\mathrm{Est}_{a}(u_j ,\hat{u}_j)$ using Lemmas \ref{lem:eig-perturbation-first}, \ref{lem:eigenvec-estimation-algorithm-1}, \ref{lem:est-eta-b-h} and \ref{lem:bbar-b-relation}.

\begin{framed}
\captionsetup{type=algorithm}
\captionof{algorithm}{}
\label{alg:eigen-errors}

\begin{algorithmic}
\State \textbf{Data:} The approximate eigen-pairs $(\hat{\lambda}_j,\hat{u}_j)$ $(j=1,\cdots,N_k+1)$ obtained by conforming FEMs, and the CR FEM eigenvalues $\lambda_{j,h}^{\CR}$.
\State \textbf{Result:} The upper bounds of the errors $\mathrm{Est}(\lambda_j,\hat{\lambda}_j)$ and  $\mathrm{Est}_{a}(u_j ,\hat{u}_j)$. 
\State 
\State \textbf{Procedures:}
\State \textbf{[Estimation of eigenvalues]}
\State Compute the range of eigenvalues $\lambda_j$ $(j=1,\cdots,N_{k}+1)$ using Lemma \ref{lem:est-tau} and Lemma \ref{thm:Lehmann--Goerisch}:
\begin{equation}
    \underline{\lambda}_j\leq\lambda_j\leq \overline{\lambda}_j~~(j=1,\cdots,N_{k}+1).
\end{equation}

\State \textbf{[Construction of Blocks]}
\State Construct blocks of exact eigenvalues
\[
\underbrace{\lambda_{n_1}= \cdots = \lambda_{N_1}}_{1\text{st block}}
<
\underbrace{\lambda_{n_2}= \cdots = \lambda_{N_2}}_{2\text{nd block}}
< \cdots <
\underbrace{\lambda_{n_k}= \cdots = \lambda_{N_k}}_{k\text{th block}}
<
\lambda_{n_{k+1}} .
\]
\State Verify the strict gaps
\[
\lambda_{N_i}\le \overline{\lambda}_{N_i}
<
\underline{\lambda}_{n_{i+1}}
\le \lambda_{n_{i+1}}
\qquad (i=1,\dots,k).
\]
\State Define
\[
E_k:=\operatorname{span}\{u_{n_k},u_{n_k+1},\dots,u_{N_k}\},
\qquad
E_k^h:=\operatorname{span}\{\hat u_{n_k},\hat u_{n_k+1},\dots,\hat u_{N_k}\}.
\]

\State

\State \textbf{[Estimation of $\delta_b(E_k,E^h_{k})$]}

Recursively obtain the values of $\mathcal{E}_b(E_k,E^h_k)$ $(k=1,\cdots k)$
using Lemma \ref{lem:eigenvec-estimation-algorithm-1} and Lemma \ref{lem:est-eta-b-h}:
\begin{equation}\label{eq:concrete-est-deltab}
    \delta_b(E_k,E^h_k)\leq \mathcal{E}_b(E_k,E^h_k) ~~(k=1,\cdots k).
\end{equation}
\For{$k=1,\ldots,K$}
  \State Compute $\varepsilon_b^h(E_l^h,E_k^h)$ for $l=1,\ldots,k-1$ by Lemma~\ref{lem:est-eta-b-h}.
  \State Set
  \[
    \theta_b^{(k)}
    :=
    \sum_{l=1}^{k-1}
    (\rho_k-\lambda_{n_l})
    \bigl[
      \varepsilon_b^h(E_l^h,E_k^h)+\mathcal{E}_b(E_l,E_l^h)
    \bigr]^2 .
  \]
  \State Obtain
  \[
    (\delta_b(E_k,E_k^h)
    \le)~~
    \mathcal{E}_b(E_k,E_k^h)
    :=
    \sqrt{
      \frac{\lambda_{N_k,h}-\lambda_{n_k}+\theta_b^{(k)}}
           {\rho_k-\lambda_{n_k}}
    } .
  \]
\EndFor

\State
\State \textbf{[Estimation of $\overline{\delta}_a(E_k,E^h_{k})$]}

Obtain the bounds for $\overline{\delta}_a(E_k,E^h_{k})$
using the estimate \eqref{eq:bbar-b-relation}:
\begin{equation*}
    \overline{\delta}^2_a(E_k,E^h_k)\leq \lambda_{N_k}+\hat\lambda_{N_k}-2\lambda_{n_k}\sqrt{1-\mathcal{E}^2_b(E_k,E^h_k)}~(=:\mathrm{Est}^2_{a}(u_j ,\hat{u}_j)).
\end{equation*}

\end{algorithmic}
\end{framed}

\medskip\medskip

We now present the details of Step 1 in the computer-assisted proof.

\subsection{Step 1: Case of $\Omega_{\mathrm{up}}$}\label{sec:case1}

In this step, we verify the optimality of the equilateral triangle in $\Omega_{\mathrm{up}}$.

\medskip

\textbf{Stationarity of the equilateral triangle.}

The following lemma shows that $p_0$ is a stationary point of both $J_1$ and $J_2$.

\begin{lemma}\label{lem:J-shape-derivative}
Let $p_0=(\tfrac12,\tfrac{\sqrt3}{2})$. Then $p_0$ is a stationary point of $J_k$ for
$k=1,2$.
\end{lemma}

\begin{proof}
Fix $e=(a,b)\in\R^2$ and set $p_t=p_0+te$.
Since $|\triangle^{(x,y)}|=y/2$, we have
\[
|\triangle^{p_0}|=\frac{\sqrt3}{4},
\qquad
\left.\frac{d}{dt}\right|_{t=0}|\triangle^{p_t}|=\frac{b}{2}.
\]
At $p_0$, the two non-horizontal sides have length $1$, and hence
\[
|\partial\triangle^{p_0}|=3,
\]
and
\[
\left.\frac{d}{dt}\right|_{t=0}|\partial\triangle^{p_t}|
=
\frac{(x-1)a+yb}{\sqrt{(x-1)^2+y^2}}\Big|_{p_0}
+\frac{xa+yb}{\sqrt{x^2+y^2}}\Big|_{p_0}
=\sqrt3\,b .
\]

It is known that the first Dirichlet eigenvalue of the equilateral triangle with side
length $1$ is
\[
\lambda_1^{p_0}=\frac{16\pi^2}{3},
\]
and a corresponding eigenfunction is given explicitly by
\begin{equation}\label{eq:explicit-u1}
U_1(x,y):=\sin\!\Bigl(\frac{4\pi}{\sqrt3}y\Bigr)
-2\sin\!\Bigl(\frac{2\pi}{\sqrt3}y\Bigr)\cos(2\pi x)
\qquad ((x,y)\in \triangle^{p_0}),
\end{equation}
which satisfies $U_1|_{\partial \triangle^{p_0}}=0$ and
$-\Delta U_1=\lambda_1^{p_0}U_1$.
Let $u_1:=U_1/\|U_1\|_{L^2(\triangle^{p_0})}$.

Applying Lemma~\ref{lem:first-order-formula} at $p_0$, we obtain
\begin{equation}\label{eq:hadamard-at-p0-explicit}
\dot\lambda_1^{p_0}
=\bigl(\dot P\nabla u_1,\nabla u_1\bigr)_{\triangle^{p_0}}
=-\frac{4}{\sqrt3}\left(
a\int_{\triangle^{p_0}} u_{1,x} u_{1,y}\,dx
+b\int_{\triangle^{p_0}} (u_{1,y})^2\,dx
\right).
\end{equation}
A direct computation using \eqref{eq:explicit-u1} yields
\begin{equation}\label{eq:two-integrals-explicit}
\int_{\triangle^{p_0}} u_{1,x}u_{1,y}\,dx=0,
\qquad
\int_{\triangle^{p_0}} (u_{1,y})^2\,dx=\frac12\,\lambda_1^{p_0}.
\end{equation}
Substituting \eqref{eq:two-integrals-explicit} into
\eqref{eq:hadamard-at-p0-explicit}, we obtain
\begin{equation}\label{eq:dotlambda-explicit}
\dot\lambda_1^{p_0}=-\frac{2}{\sqrt3}\,b\,\lambda_1^{p_0}.
\end{equation}

\paragraph{The case $k=1$.}
Since
\[
J_1(\triangle^{p})
=\lambda_1(\triangle^{p})|\triangle^{p}|
-\frac{\pi^2}{16}\frac{|\partial\triangle^{p}|^2}{|\triangle^{p}|}
-\frac{7\sqrt3\pi^2}{12},
\]
we have
\[
\left.\frac{d}{dt}\right|_{t=0}J_1(\triangle^{p_t})
=
\dot\lambda_1^{p_0}|\triangle^{p_0}|
+\lambda_1^{p_0}\left.\frac{d}{dt}\right|_{t=0}|\triangle^{p_t}|
-\frac{\pi^2}{16}
\left.\frac{d}{dt}\right|_{t=0}
\frac{|\partial\triangle^{p_t}|^2}{|\triangle^{p_t}|}.
\]
Using
\[
|\triangle^{p_0}|=\frac{\sqrt3}{4},
\qquad
\left.\frac{d}{dt}\right|_{t=0}|\triangle^{p_t}|=\frac{b}{2},
\]
and \eqref{eq:dotlambda-explicit}, the first two terms cancel:
\[
\dot\lambda_1^{p_0}|\triangle^{p_0}|
+\lambda_1^{p_0}\left.\frac{d}{dt}\right|_{t=0}|\triangle^{p_t}|
=
\Bigl(-\frac{2}{\sqrt3}b\lambda_1^{p_0}\Bigr)\frac{\sqrt3}{4}
+\lambda_1^{p_0}\frac{b}{2}
=0.
\]
Moreover, using
\[
|\partial\triangle^{p_0}|=3,
\qquad
\left.\frac{d}{dt}\right|_{t=0}|\partial\triangle^{p_t}|=\sqrt3\,b,
\]
together with the above identities for $|\triangle^{p_0}|$ and
$\left.\frac{d}{dt}\right|_{t=0}|\triangle^{p_t}|$, we directly obtain
\[
\left.\frac{d}{dt}\right|_{t=0}
\frac{|\partial\triangle^{p_t}|^2}{|\triangle^{p_t}|}=0.
\]
Therefore,
\[
\left.\frac{d}{dt}\right|_{t=0}J_1(\triangle^{p_t})=0.
\]

\paragraph{The case $k=2$.}
Set $C_*=\frac{4\pi^2}{(3+\sqrt{\pi\sqrt3})^2}$. Then
\[
J_2(\triangle^{p_t})
=
\lambda_1(\triangle^{p_t})|\triangle^{p_t}|
-
C_*\frac{\bigl(|\partial\triangle^{p_t}|+\sqrt{4\pi|\triangle^{p_t}|}\bigr)^2}
{4|\triangle^{p_t}|}.
\]
As above, the derivative of $\lambda_1(\triangle^{p_t})|\triangle^{p_t}|$ vanishes at
$t=0$. Using
\[
|\triangle^{p_0}|=\frac{\sqrt3}{4},
\qquad
\left.\frac{d}{dt}\right|_{t=0}|\triangle^{p_t}|=\frac{b}{2},
\qquad
|\partial\triangle^{p_0}|=3,
\qquad
\left.\frac{d}{dt}\right|_{t=0}|\partial\triangle^{p_t}|=\sqrt3\,b,
\]
we directly obtain
\[
\left.\frac{d}{dt}\right|_{t=0}
\frac{\bigl(|\partial\triangle^{p_t}|+\sqrt{4\pi|\triangle^{p_t}|}\bigr)^2}
{4|\triangle^{p_t}|}
=0.
\]
Therefore,
\[
\left.\frac{d}{dt}\right|_{t=0}J_2(\triangle^{p_t})=0.
\]
Since $e$ was arbitrary, $p_0$ is a stationary point of both $J_1$ and $J_2$.
\end{proof}

\medskip

We next evaluate the second-order directional derivatives of $J_k$ $(k=1,2)$ in the $x$- and $y$-directions. To derive lower bounds for $\frac{\partial^2 J_k}{\partial x^2}$ and $\frac{\partial^2 J_k}{\partial y^2}$ $(k=1,2)$, we use the following relations between the derivatives of $J_k$ and those of the eigenvalue $\lambda_1(\triangle^p)$.

Set
\[
r_1:=\sqrt{x^{2}+y^{2}},\qquad
r_2:=\sqrt{(x-1)^{2}+y^{2}}.
\]

\begin{lemma}\label{lem:Ji-xx-yy-simpler-nopA}
The second partial derivatives of $J_k(\triangle^p)$ $(k=1,2)$ satisfy
\begin{align}
\frac{\partial^2 J_k}{\partial x^2}(x,y)
&= \frac{1}{2}\,y\,\frac{\partial^{2}\lambda_{1}^{p}}{\partial x^{2}}
   + R^{(k)}_{xx}(x,y),
\label{eq:Jkxx-simple}\\[1ex]
\frac{\partial^2 J_k}{\partial y^2}(x,y)
&= \frac{1}{2}\,y\,\frac{\partial^{2}\lambda_{1}^{p}}{\partial y^{2}}
   + \frac{\partial\lambda_{1}^{p}}{\partial y}
   + R^{(k)}_{yy}(x,y),
\label{eq:Jkyy-simple}
\end{align}
where $R^{(k)}_{xx}$ and $R^{(k)}_{yy}$ are given as follows.

\medskip\noindent
\emph{(i) The case $k=1$.}
We have
\begin{align}
R^{(1)}_{xx}(x,y)
&=
-\frac{\pi^{2}}{4y}
\left[
\left(\frac{x}{r_1}+\frac{x-1}{r_2}\right)^{2}
+(1+r_1+r_2)\,y^{2}\left(\frac{1}{r_1^{3}}+\frac{1}{r_2^{3}}\right)
\right],
\label{eq:Rxx1-xy}\\[0.5ex]
R^{(1)}_{yy}(x,y)
&=
-\frac{\pi^{2}}{4y^{3}}
\left[
y^{2}(1+r_1+r_2)\left(\frac{x^{2}}{r_1^{3}}+\frac{(x-1)^{2}}{r_2^{3}}\right)
+\left(1+\frac{x^{2}}{r_1}+\frac{(x-1)^{2}}{r_2}\right)^{2}
\right].
\label{eq:Ryy1-xy}
\end{align}

\medskip\noindent
\emph{(ii) The case $k=2$.}
Set
\[
Q(x,y):=1+r_1+r_2+2\sqrt{\frac{\pi y}{2}},
\]
\[
Q_y(x,y):=\frac{y}{r_1}+\frac{y}{r_2}+\sqrt{\frac{\pi}{2y}},
\qquad
Q_{yy}(x,y):=\frac{x^{2}}{r_1^{3}}+\frac{(x-1)^{2}}{r_2^{3}}
-\frac12\sqrt{\frac{\pi}{2}}\,y^{-3/2}.
\]
Then
\begin{align}
R^{(2)}_{xx}(x,y)
&=
-\frac{4\pi^{2}}{(3+\sqrt{\pi\sqrt3})^{2}}\,
\frac{1}{y}
\left[
\left(\frac{x}{r_1}+\frac{x-1}{r_2}\right)^{2}
+Q(x,y)\,y^{2}\left(\frac{1}{r_1^{3}}+\frac{1}{r_2^{3}}\right)
\right],
\label{eq:Rxx2-xy}\\[0.5ex]
R^{(2)}_{yy}(x,y)
&=
-\frac{2\pi^{2}}{(3+\sqrt{\pi\sqrt3})^{2}}
\left[
\frac{2}{y}\Bigl(Q_y(x,y)^{2}+Q(x,y)\,Q_{yy}(x,y)\Bigr)
-\frac{4}{y^{2}}Q(x,y)\,Q_y(x,y)
+\frac{2}{y^{3}}Q(x,y)^{2}
\right].
\label{eq:Ryy2-xy}
\end{align}
\end{lemma}

\begin{proof}
The identities follow from direct computation.
\end{proof}

\textbf{Step 1--2: Convexity in the \(x\)-direction}

Define the grid points by
\[
p_{ij}
=
\left(
\frac{1}{2}+2\varepsilon_{\mathrm{up}}\frac{i}{N_x},
\,
\frac{\sqrt3}{2}-\varepsilon_{\mathrm{up}}\frac{j}{N_y}
\right),
\qquad
i=0,1,\ldots,N_x,
\quad
j=0,1,\ldots,N_y.
\]
For \(i=0,1,\ldots,N_x-1\) and \(j=0,1,\ldots,N_y-1\), let \(R_{ij}\) denote the rectangle with vertices
\[
p_{ij},
\quad
p_{i+1,j},
\quad
p_{i,j+1},
\quad
p_{i+1,j+1}.
\]
Then \(\{R_{ij}\}\) forms a partition of
\[
\left[\frac{1}{2},\frac{1}{2}+2\varepsilon_{\mathrm{up}}\right]
\times
\left[\frac{\sqrt3}{2}-\varepsilon_{\mathrm{up}},\frac{\sqrt3}{2}\right],
\]
and hence, in particular, covers \(\Omega_{\mathrm{up}}\).
In the computations below, we take
\[
\varepsilon_{\mathrm{up}}=0.122,
\qquad
N_x=20,
\qquad
N_y=100.
\]

To obtain lower bounds for
\[
\frac{\partial^2 J_k}{\partial x^2}
\qquad
(k=1,2)
\]
on each \(R_{ij}\), we apply the lower estimate for the second-order shape derivative from Theorem~\ref{thm:main-theorem-est} together with Lemma~\ref{lem:Ji-xx-yy-simpler-nopA}, as implemented in the following algorithm.

\begin{framed}
  \captionsetup{type=algorithm}
  \captionof{algorithm}{Certification of the sign of $\frac{\partial^2 J_k}{\partial x^2}$ $(k=1,2)$}
  \label{alg:certify-d2Jdx2}

  \begin{algorithmic}
    \State \textbf{Input:} Rectangular domain $R_{ij}$
    \State \textbf{Output:} Rigorous lower bounds $L_{ij}^{(k)}$ for $\frac{\partial^2 J_k}{\partial x^2}$ $(k=1,2)$ on $R_{ij}$

    \State
    \State \textbf{[Eigenvalue estimation]}

    \State Over $\triangle^{p_{ij}}$, compute enclosures of the eigenvalues $\lambda_k^{p_{ij}}$ $(k=1,\cdots,N_{\texttt{spec}}+1)$ using Lemma \ref{lem:est-tau}.

    \State Using \eqref{eq:perturbation-first} with $p=p_{ij}$ and $\tilde p\in R_{ij}$, estimate the range of $\lambda_k^p$ $(k=1,\cdots,N_{\texttt{spec}}+1)$ over $R_{ij}$ and obtain the uniform bound
    \begin{equation}\label{eq:uniform-estimate-lam-ij}
      \lambda_{\min}\left([S^{-1}S^{-\intercal}]\right)\cdot \lambda_k^{p_{ij}}
      \le \lambda_k^{\tilde p}
      \le \lambda_{\max}\left([S^{-1}S^{-\intercal}]\right)\cdot \lambda_k^{p_{ij}},
    \end{equation}
    where the interval matrix $[S^{-1}S^{-\intercal}]$ is defined by
    \begin{align*}
      [S^{-1}S^{-\intercal}]
      :=\{S^{-1}S^{-\intercal}\mid S:\triangle^{p_{ij}}\to \triangle^{\tilde p}
      \text{ is a linear map},\ \tilde p\in R_{ij}\}.
    \end{align*}
    Here, $\lambda_{\min}(\cdot)$ and $\lambda_{\max}(\cdot)$ denote the minimum and maximum eigenvalues of an interval matrix, respectively. Let
    \begin{equation*}
      [\lambda_k]
      :=\left\{x\in\mathbb{R}\,\middle|\,
      \lambda_{\min}\left([S^{-1}S^{-\intercal}]\right)\cdot\lambda_k^{p_{ij}}
      \le x \le
      \lambda_{\max}\left([S^{-1}S^{-\intercal}]\right)\cdot\lambda_k^{p_{ij}}
      \right\}
      \qquad (k=1,\cdots,N_{\texttt{spec}}+1).
    \end{equation*}

    \State Verify that $\lambda_{N_{\texttt{spec}}}^p$ and $\lambda_{N_{\texttt{spec}}+1}^p$ are separated for every $p\in R_{ij}$, namely,
    \begin{equation} [\lambda_{N_{\texttt{spec}}}]<[\lambda_{N_{\texttt{spec}}+1}].
    \end{equation}

    \State
    \State \textbf{[Eigenfunction estimation]}

    \State Over $\triangle^{p_{ij}}$, compute the approximate eigenfunctions
    \[
      \hat{u}_k^{p_{ij}}
      = \alpha_{1,k}\phi_1^{p_{ij}}+\cdots+\alpha_{N_{\texttt{max}},k}\phi_{N_{\texttt{max}}}^{p_{ij}}
      \qquad (k=1,\cdots,N_{\texttt{spec}}+1),
    \]
    where $\phi_1^{p_{ij}},\cdots,\phi_{N_{\texttt{max}}}^{p_{ij}}$ form a basis of the FEM space $V_h(\triangle^{p_{ij}})$ and $(\alpha_{1,k},\cdots,\alpha_{N_{\texttt{max}},k})$ is the corresponding coefficient vector.

    \State For each $p\in R_{ij}$, define the approximate eigenfunctions on $\triangle^p$ by
    \[
      \hat{u}_k^p
      := \alpha_{1,k}\phi_1^p+\cdots+\alpha_{N_{\texttt{max}},k}\phi_{N_{\texttt{max}}}^p
      \qquad (k=1,\cdots,N_{\texttt{spec}}+1).
    \]

    \State Define the approximate eigenvalues $\hat{\lambda}_k^p$ $(k=1,\cdots,N_{\texttt{spec}})$ by
    \begin{equation*}
      \hat{\lambda}_k^p
      :=\min_{\widehat{V}_{\texttt{spec}}^{(k)}\subset \widehat{V}_{\texttt{spec}}}
      \max_{\hat v\in \widehat{V}_{\texttt{spec}}^{(k)}}
      \frac{\|\nabla \hat v\|_{\triangle^p}^2}{\|\hat v\|_{\triangle^p}^2},
    \end{equation*}
    where
    \[
      \widehat{V}_{\texttt{spec}}
      := \mathrm{span}\{\hat{u}_1^p,\cdots,\hat{u}_{N_{\texttt{spec}}}^p\},
    \]
    and $\widehat{V}_{\texttt{spec}}^{(k)}$ is a $k$-dimensional subspace of $\widehat{V}_{\texttt{spec}}$. Let
    \[
      [\hat{\lambda}_k]:=\bigcup_{p\in R_{ij}}\hat{\lambda}_k^p.
    \]

    \State Apply Algorithm \ref{alg:eigen-errors} to obtain uniform upper bounds for $\mathrm{Est}_a(u_k^p,\hat{u}_k^p)$ for $p\in R_{ij}$ and $k=1,\cdots,N_{\texttt{spec}}+1$. Denote by $\overline{\mathrm{Est}_a(u_k^p,\hat{u}_k^p)}$ the corresponding uniform upper bound on $R_{ij}$.

    \State
    \State \textbf{[Lower bound for $\frac{\partial^2 \lambda_1^p}{\partial x^2}$]}

    \State Using \eqref{eq:uniform-estimate-lam-ij}, compute the range of
    \begin{equation}
      \widehat{(\ddot\lambda_{1,N_{\texttt{spec}}}^p)}
      =
      \left(\ddot P\nabla \hat u_1^p,\nabla \hat u_1^p\right)_{\triangle^p}
      +2\sum_{k=2}^{N_{\texttt{spec}}}
      \frac{\left|\ip{\dot P\nabla \hat u_1^p}{\nabla \hat u_k^p}\right|^2}{\lambda_k^p-\lambda_1^p}
    \end{equation}
    over $R_{ij}$. Let
    \[
      [\widehat{(\ddot\lambda_{1,N_{\texttt{spec}}}^p)}]
      := \bigcup_{p\in R_{ij}}
      \left\{\widehat{(\ddot\lambda_{1,N_{\texttt{spec}}}^p)}\right\}.
    \]

    \State Let $[\dot P]$ and $[\ddot P]$ be the interval matrices defined by
    \[
      [\dot P]:=\bigcup_{p\in E_{ij}}\dot P,
      \qquad
      [\ddot P]:=\bigcup_{p\in E_{ij}}\ddot P.
    \]

    \State Apply Theorem \ref{thm:main-theorem-est} to obtain the uniform lower bound
    \begin{align}\label{eq:main-theorem-est-applied-1}
      \frac{\partial^2 \lambda_1^p}{\partial x^2}
      \ge\;&
      [\widehat{(\ddot\lambda_{1,N_{\texttt{spec}}}^p)}]
      -\|[\ddot P]\|_2\left(\sqrt{[\lambda_1]}+\sqrt{[\hat\lambda_1]}\right)
      \overline{\mathrm{Est}_a(u_1^p,\hat{u}_1^p)} \notag\\
      &\quad
      -2\sum_{k=2}^{N_{\texttt{spec}}}
      \frac{[\mathcal{A}_{1k}]}{|[\lambda_k]-[\lambda_1]|}
      \qquad (=: [L_{xx}]),
    \end{align}
    where
    \begin{align*}
      [\mathcal{A}_{1k}]
      :=\|[\dot P]\|_2^2
      \left(
      \overline{\mathrm{Est}_a(u_1^p,\hat{u}_1^p)}\sqrt{[\hat\lambda_k]}
      +\sqrt{[\lambda_1]}\,\overline{\mathrm{Est}_a(u_k^p,\hat{u}_k^p)}
      \right)
      \left(
      \sqrt{[\lambda_1][\lambda_k]}+\sqrt{[\hat\lambda_1][\hat\lambda_k]}
      \right).
    \end{align*}

    \State
    \State \textbf{[Lower bound for $\frac{\partial^2 J_k}{\partial x^2}$]}

    \State Obtain the uniform lower bound for $\frac{\partial^2 J_k}{\partial x^2}$ on $R_{ij}$ by
    \begin{equation*}
      \frac{\partial^2 J_k}{\partial x^2}(x,y)
      \ge
      \frac{1}{2}\inf \bigl([y_j,y_{j+1}]\,[L_{xx}]\bigr)
      + \inf_{(x,y)\in R_{ij}} R_{xx}(x,y)
      \qquad (=:L_{ij}^{(k)}).    \end{equation*}
  \end{algorithmic}
\end{framed}

\medskip
With the choice \(N_{\mathrm{spec}}=1\), we obtain
\[
  \inf_{i,j} L_{ij}^{(1)} \ge 10.65 \ (>0), \qquad
  \inf_{i,j} L_{ij}^{(2)} \ge 10.66 \ (>0).
\]
Hence,
\[
  \frac{\partial^2 J_k}{\partial x^2} > 0
  \qquad \text{in } \Omega_{\mathrm{up}}~~(k=1,2),
\]
which shows that \(J_k\) is strictly convex in the \(x\)-direction on \(\Omega_{\mathrm{up}}\) for \(k=1,2\).
By symmetry, the minimizer in \(\Omega_{\mathrm{up}}\) lies on the symmetry axis \(x=\tfrac12\).

\begin{remark}
In the above setting, we choose $N_{\mathrm{spec}}=1$,
since error bounds for higher eigenfunctions may become unstable for the clustered eigenvalues in the neighborhood of $p_0$. When eigenvalues are clustered, the corresponding
eigenfunctions are highly sensitive to perturbations, and the estimates of
Liu--Vejchodsk\'y~\cite{liu2022fully} may be too large to verify positivity of
the second-order derivative. The same issue may occur in the neighborhood of the diabolical
points of $\Omega$, where eigenvalues nearly coincide; see
Berry--Wilkinson~\cite{berry1984diabolical}. This choice $N_{\mathrm{spec}}=1$ avoids the unstable
higher modes, while the omitted tail of the spectral expansion is nonnegative
and therefore may be discarded without affecting the rigorous lower bound.
\end{remark}

\medskip

\textbf{Step 1-3: Convexity in the $y$-direction}

Next, we estimate the sign of $\frac{\partial^2 J_k}{\partial y^2}$ over the segment $I^{y}$ defined by
\[I^{y}:=\left\{(x,y)\in \mathbb{R}^2~:~
   x=1/2,~~y\in \Bigl[\frac{\sqrt3}{2}-\varepsilon_{\mathrm{up}},\ \frac{\sqrt3}{2}\Bigr]
   \right\}.
\]
For $m=1,\ldots,N_y+1$, let
\begin{equation}
    y_m
    =
    \frac{\sqrt3}{2}-\varepsilon_{\mathrm{up}}
    +\frac{m-1}{N_y}\varepsilon_{\mathrm{up}},
    \qquad
    p_m:=(1/2,y_m).
\end{equation}
Let $\{I^{y}_m\}_{m=1}^{N_y}$ be the subdivision of $I^y$ defined by
\[
I^{y}_m
:=
\left\{(x,y)\in \mathbb{R}^2~:~
   x=1/2,\quad y\in [y_m,y_{m+1}]
   \right\}
\qquad (m=1,\ldots,N_y),
\]
with $\varepsilon_{\mathrm{up}}=0.122$ and $N_y=200$.

\medskip

\begin{framed}
  \captionsetup{type=algorithm}
  \captionof{algorithm}{Certification of the sign of $\frac{\partial^2 J_k}{\partial y^2}$}
  \label{alg:certify-d2Jdy2}

  \begin{algorithmic}
    \State \textbf{Data:} Interval $I_m^y$
    \State \textbf{Result:} Rigorous lower bounds $L_{m}^{(k)}$ for $\frac{\partial^2 J_k}{\partial y^2}$ over each $I_m^y$
    
    \State 
    \State \textbf{[Estimation of eigenvalues]}
    
    \State Over $\triangle^{p_{m}}$, compute the range of eigenvalues $\lambda^{p_{m}}_k$ $(k=1,\cdots,N_{\texttt{spec}}+1)$ using Lemma \ref{lem:est-tau}.
    
    \State Estimate the range of $\lambda_k^{p}$ $(k=1,\cdots,N_{\texttt{spec}}+1)$ over the interval $I_m^y$ using \eqref{eq:perturbation-first} with $p=p_{m}$ and $\tilde{p}\in I_m^y$ and the obtain the uniform bound for all $\lambda_k^{\tilde{p}}$ $(\tilde{p}\in I_m^y)$:
    \begin{equation}\label{eq:uniform-estimate-lam-m}
    \lambda_{\min}\left([S^{-1}S^{-\intercal}]\right)\cdot\lambda^{p_{m}}_k\leq \lambda^{\tilde p}_k\leq\lambda_{\max}\left([S^{-1}S^{-\intercal}]\right)\cdot\lambda^{p_{m}}_k,
    \end{equation}
    where the interval matrix $[S^{-1}S^{-\intercal}]$ is defined by
    \begin{align*}
     [S^{-1}S^{-\intercal}]&:=\{S^{-1}S^{-\intercal}~|~S:T^{p_{m}}\to T^{\tilde p} \mbox{ being a linear map},~~\tilde{p}\in I_m^y\}.
    \end{align*}
    Here, $\lambda_{\min}(\cdot)$ and $\lambda_{\max}(\cdot)$ take the minimum and the maximum eigenvalues of an interval matrix, respectively.
    Let
    \begin{equation*}
    [\lambda_k]:=\{x\in\mathbb{R}~|~\lambda_{\min}\left([S^{-1}S^{-\intercal}]\right)\cdot\lambda^{p_{m}}_k\leq x\leq \lambda_{\max}\left([S^{-1}S^{-\intercal}]\right)\cdot\lambda^{p_{m}}_k\}.
    \end{equation*}
    \State Confirm that $\lambda_{N_{\texttt{spec}}}^p$ and $\lambda_{N_{\texttt{spec}+1}}^p$ are separated for $p\in I_m^y$, i.e.,
\begin{equation}
    [\lambda_{N_{\texttt{spec}}}]<[\lambda_{N_{\texttt{spec}+1}}]
\end{equation}
holds.

    \State 
    \State \textbf{[Estimation of eigenfunctions]}
    
    \State Over $\triangle^{p_{m}}$,  compute the approximate eigenfunctions 
    $$
    \hat{u}_k^{p_{m}} =\alpha_{1,k}\phi_{1}^{p_{m}}+\cdots +\alpha_{N_{\texttt{max}},k}\phi_{N_{\texttt{max}}}^{p_{m}}~~(k=1,\cdots,N_{\texttt{spec}}+1),
    $$
    where $\phi_{1}^{p_{m}},\cdots,\phi_{N_{\texttt{max}}}^{p_{m}}$ are the basis of FEM space $V_h(\triangle^{p_{m}})$ and $(\alpha_{1,k},\cdots,\alpha_{N_{\texttt{max}},k})$ is the coefficient vector.
    
    \State 
    \State For $p\in I_m^y$, define the approximate eigenfunctions $\hat{u}^p_k$ over $\triangle^p$ by    
    $$\hat{u}^p_k:=\alpha_{1,k}\phi_{1}^{p}+\cdots +\alpha_{N_{\texttt{max}},k}\phi_{N_{\texttt{max}}}^{p}~~(k=1,\cdots,N_{\texttt{spec}}+1).
    $$

    \State Define the approximated eigenvalues $\hat{\lambda}_k^p$ $(k=1,\cdots,N_{\texttt{spec}})$ by
    \begin{equation*}        \hat{\lambda}_k^p:=\min_{\widehat{V}^{(k)}_{\texttt{spec}}\subset \widehat{V}_{\texttt{spec}}}\max_{\hat v\in \widehat{V}^{(k)}_{\texttt{spec}} }\frac{\|\nabla \hat{v}\|^2_{\triangle^p}}{\|\hat{v}\|^2_{\triangle^p}},
    \end{equation*}
    where $\widehat{V}_{\texttt{spec}}:=\mbox{span}\{\hat{u}^p_1,\cdots,\hat{u}^p_{N_{\texttt{spec}}}\}$ and $\widehat{V}^{(k)}_{\texttt{spec}}$ is its $k$-dimensional subspace. Let $[\hat{\lambda}_k]:=\cup_{p\in I_m^y}\hat{\lambda}_k^p$.

    \State Apply Algorithm \ref{alg:eigen-errors} and obtain uniform upper bounds of $\mathrm{Est}_{a}(u_k^p ,\hat{u}_k^p)$ for $p\in I_m^y$ and $k=1,\cdots,N_{\texttt{spec}+1}$.
    Let $\overline{\mathrm{Est}_{a}(u_k^p ,\hat{u}_k^p)}$ be the uniform upper bound of $\mathrm{Est}_{a}(u_k^p ,\hat{u}_k^p)$ on $I_m^y$.
  
  \State 
  \State \textbf{[Lower bound of $\frac{\partial^2 \lambda_1^p}{\partial y^2}$]}
  \State Using the estimate \eqref{eq:uniform-estimate-lam-m}, obtain the range of
  \begin{equation}    \widehat{(\ddot\lambda^p_{1,N_{\texttt{spec}}})} = \left(\ddot P\nabla \hat u_1^p,\nabla \hat u_1^p\right)_{\triangle^p} + 2 \sum_{\substack{k=2}}^{N_{\texttt{spec}}} \frac{\left|\ip{\dot{P}\nabla\hat  u_1^p}{\nabla \hat u_k^p}\right|^2}{\lambda_k^p - \lambda_1^p}
\end{equation}
over $I_m^y$. Let $[\widehat{(\ddot\lambda^p_{1,N_{\texttt{spec}}})}]:=\cup_{p\in I_m^y}\left\{\widehat{(\ddot\lambda^p_{1,N_{\texttt{spec}}})}\right\}$.

  \State Let $[\dot P],[\ddot P]$ be the interval matrix defined by $[\dot P]:=\bigcup_{p\in E_{m}}\dot{P}$ and $[\ddot P]:=\bigcup_{p\in E_{m}}\ddot{P}$.
  
  \State Use Theorem \ref{thm:main-theorem-est} to obtain the uniform lower bound of $\frac{\partial^2 \lambda_1^p}{\partial y^2}$:
  \begin{align}\label{eq:main-theorem-est-applied-2}
    \frac{\partial^2 \lambda_1^p}{\partial y^2} \ge & ~~[\widehat{(\ddot\lambda^p_{1,N_{\texttt{spec}}})}]-\left\|[\ddot P]\right\|_{2} \left(\sqrt{[\lambda_1]}+\sqrt{[\hat\lambda_1]}\right) \overline{\mathrm{Est}_{a}(u_1^p ,\hat{u}_1^p)}  - 2 \sum_{\substack{k=2}}^{N_{\texttt{spec}}}  \frac{[\mathcal{A}_{1k}]}{|[\lambda_k] - [\lambda_1]|}~(=:[L_{yy}]),
\end{align}
where the terms $[\mathcal{A}_{1k}]$ are given by
    \begin{align*}
      [\mathcal{A}_{1k}]
      :=\|[\dot P]\|_2^2
      \left(
      \overline{\mathrm{Est}_a(u_1^p,\hat{u}_1^p)}\sqrt{[\hat\lambda_k]}
      +\sqrt{[\lambda_1]}\,\overline{\mathrm{Est}_a(u_k^p,\hat{u}_k^p)}
      \right)
      \left(
      \sqrt{[\lambda_1][\lambda_k]}+\sqrt{[\hat\lambda_1][\hat\lambda_k]}
      \right).
    \end{align*}

\State \textbf{[Lower bound of $\frac{\partial \lambda_1^p}{\partial y}$]}

\State Compute the discrete first-order shape derivative in the $y$-direction
\begin{equation*}
\widehat{(\dot\lambda_1^p)}
:=\bigl(\dot P\nabla \hat u_1^p,\nabla \hat u_1^p\bigr)_{\triangle^p}
\qquad\text{with}\qquad
\dot P=\begin{pmatrix}0 & 0\\ 0 & -2/y\end{pmatrix},
\end{equation*}
and let $[\widehat{(\dot\lambda_1^p)}]:=\bigcup_{p\in I_m^y}\{\widehat{(\dot\lambda_1^p)}\}$.
Apply Lemma~\ref{lem:first-order-error-estimate} to obtain the uniform lower bound
\begin{equation}\label{eq:first-order-est-applied}
  \frac{\partial \lambda_1^p}{\partial y}
  \;\ge\;
  [\widehat{(\dot\lambda_1^p)}]
  -\|[\dot P]\|_2\Bigl(\sqrt{[\lambda_1]}+\sqrt{[\hat\lambda_1]}\Bigr)
  \overline{\mathrm{Est}_a(u_1^p,\hat u_1^p)}
  \qquad (=:[L_y]).
\end{equation}

\State \textbf{[Lower bound of $\frac{\partial^2 J_k}{\partial y^2}$]}

\State Use Lemma \ref{lem:Ji-xx-yy-simpler-nopA} together with~\eqref{eq:first-order-est-applied} to obtain the uniform lower bound for $\frac{\partial^2 J_k}{\partial y^2}$ over $I_m^y$. Then, we obtain
\begin{equation*}
\frac{\partial^2 J_k}{\partial y^2}(x,y)
\geq\frac{1}{2}\inf\!\bigl([y_m,y_{m+1}]\,[L_{yy}]\bigr)
+\inf[L_y] \;+\; \inf_{(x,y)\in I_m^y} R_{yy}^{(k)}(x,y)~(=:L_{m}^{(k)}).
\end{equation*}
\end{algorithmic}
\end{framed}

\medskip

With the choice $N_{\mathrm{spec}}=1$, the rigorous run produced
\[
  \inf_{m}L_{m}^{(1)}\ \ge 3.46 ~(> 0),~~\inf_{m}L_{m}^{(2)}\ \ge 3.27 ~(> 0).
\]
Hence,
\[
\frac{\partial^2 J_k}{\partial y^2} > 0
\quad \text{on } I^y \quad (k=1,2),
\]
which shows that \(J_k\) is strictly convex in the \(y\)-direction on \(I^y\).  
By combining this with the result in Step 1-1, $J_k$ takes its minimum at $(x,y)=(1/2,\sqrt{3}/2)$ in $\Omega_{\mathrm{up}}$.

\medskip\medskip

\subsection{Step 2: Case of $\Omega_{\mathrm{mid}}$}\label{sec:case2}

We estimate the uniform lower bound of the shape functional $J$ over $\Omega_{\mathrm{mid}}$.
Each point $p=(x,y) (\in \Omega_{\mathrm{mid}})$ is uniquely parametrized by $(x, \theta) \in [1/2, 1) \times (0,\pi/3)$ through the relation:
\[
    \theta(x,y) := \arctan(y/x).
\]


To capture the behavior of eigenvalues, we construct a non-uniform grid points  $p_{i,j}=(x_{i},x_{i}\tan\theta_{j})(\in\Omega_{\mathrm{mid}})$.
For a pair $(x,\theta)$, we define a cell $C_{x\theta}$ as the product interval
\[
    C_{x\theta} := [x_{i}, x_{i+1}] \times \left[\theta_j, \theta_{j+1}\right].
\]
Correspondingly, let $C_{xy}$ denote the quadrilateral subregion with vertices
$$
(x_{i},x_{i}\tan\theta_{j}),~
(x_{i},x_{i}\tan\theta_{j+1}),~
(x_{i+1},x_{i+1}\tan\theta_{j}),~
(x_{i+1},x_{i+1}\tan\theta_{j+1}).
$$
Then, the collection $\{C_{xy}\}$ forms a covering of $\Omega_\mathrm{mid}$.
The precise definitions of these cells, along with the associated finite element parameters such as mesh sizes and polynomial orders for each cell, are provided in the dataset \texttt{inputs/cell\_def.csv} in the code repository~\cite{endo2025code}.


Due to the geometric construction of the triangle $T^p$, for any $p \in C_{xy}$, the following domain inclusion holds:
\[
     T^{p_{i,j}}\subset T^p \subset T^{p_{i+1,j+1}}
\]
as illustrated in Figure~\ref{fig:cell-monotonicity}.
Consequently, the monotonicity property of Dirichlet eigenvalues implies
\begin{equation}\label{eq:cell-monotonicity}
    \lambda_k(T^{p_{i+1,j+1}}) \le \lambda_k(T^p) \le \lambda_k(T^{p_{i,j}}), \quad \forall p \in C_{xy}.
\end{equation}
\begin{figure}[H]
    \centering
    \begin{tikzpicture}[scale=1.0]

  \coordinate (O) at (0,0);    
  \coordinate (B) at (6,0);    
  \coordinate (T) at (5,3);    
  \coordinate (P) at (2.5,0.6);
  \coordinate (Q) at (3.0,1.2);

  \coordinate (A) at (2.5,1.5); 
  \coordinate (D) at (5,1.2);   

  \draw (O) -- (B) -- (T) -- cycle;

  \draw (O) -- (B) -- (Q) -- cycle;

  \draw (O) -- (P) -- (B);

  \draw[very thick]
        (P) -- (A) -- (T) -- (D) -- cycle;

  \node[below left]  at (O) {$\left(0,0\right)$};
  \node[below right] at (B) {$\left(1,0\right)$};

  \node[below]       at (P) {$p_{i,j}$};
  \node[above right] at (T) {$p_{i+1,j+1}$};
  \node[above right] at (Q) {$p$};

  \node at (4.0,1.5) {$C_{i,j}$};

\end{tikzpicture}
      \caption{Parameter cell $\mathcal{C}_{i,j}$.}
      \label{fig:cell-monotonicity}
\end{figure}
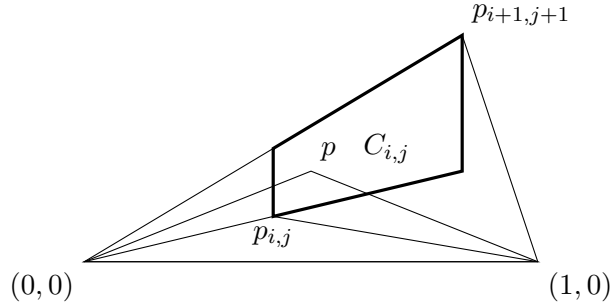

For $p=(x,y)$ and $\Lambda\in\mathbb{R}$, we introduce auxiliary functionals
$\mathcal{B}_1,\mathcal{B}_2$ corresponding to $J_1,J_2$ by
\begin{align}
\label{eq:def-B1}
\mathcal{B}_1(p;\Lambda)
&:= \Lambda\,|\triangle^p|
-\frac{\pi^2}{16}\frac{|\partial\triangle^p|^2}{|\triangle^p|}
-\frac{7\sqrt{3}\pi^2}{12},\\
\label{eq:def-B2}
\mathcal{B}_2(p;\Lambda)
&:= \Lambda\,|\triangle^p|
-\frac{4\pi^2}{\left(3+\sqrt{\pi\sqrt3}\right)^2}\,
\frac{\left(|\partial\triangle^p|+\sqrt{4\pi|\triangle^p|}\right)^2}{4~|\triangle^p|}.
\end{align}
Note that $J_k(\triangle^p)=\mathcal{B}_k\!\left(p;\lambda_1(\triangle^p)\right)$ for $k=1,2$.

\begin{lemma}\label{lem:bounds-over-Rij}
Let $\mathcal{C}_{ij}$ be a verification cell in $\Omega_{\mathrm{mid}}$ with vertices
$p_{i,j},\,p_{i+1,j},\,p_{i,j+1},\,p_{i+1,j+1}$ as in the construction above.
Then for each $k\in\{1,2\}$ and every $p\in\mathcal{C}_{ij}$ we have the uniform bounds
\begin{align}
\label{eq:bounds-over-Cij-Jk}
\mathcal{B}_k\!\left(p_{i,j};\,\lambda_1(\triangle^{p_{i+1,j+1}})\right)
\;\le\;
J_k(\triangle^p)
\;\le\;
\mathcal{B}_k\!\left(p_{i+1,j+1};\,\lambda_1(\triangle^{p_{i,j}})\right).
\end{align}
\end{lemma}

\begin{proof}
Fix $k\in\{1,2\}$ and let $p=(x,y)\in\mathcal{C}_{ij}$.
By the geometric construction of the cells (cf.\ \eqref{eq:cell-monotonicity}),
\[
\triangle^{p_{i,j}}\subset \triangle^{p}\subset \triangle^{p_{i+1,j+1}}.
\]
Hence, by domain monotonicity of the first Dirichlet eigenvalue,
\begin{equation}\label{eq:eig-mono-k}
\lambda_1(\triangle^{p_{i+1,j+1}})
\le
\lambda_1(\triangle^{p})
\le
\lambda_1(\triangle^{p_{i,j}}).
\end{equation}
Since $\mathcal{B}_k(p;\Lambda)$ is affine in $\Lambda$ with positive coefficient
$|\triangle^p|=y/2>0$, the map
\[
\Lambda\longmapsto \mathcal{B}_k(p;\Lambda)
\]
is strictly increasing for each fixed $p$.
Applying this to \eqref{eq:eig-mono-k}, we obtain
\begin{equation}\label{eq:Bk-eig-sandwich}
\mathcal{B}_k\!\left(p;\lambda_1(\triangle^{p_{i+1,j+1}})\right)
\le
\mathcal{B}_k\!\left(p;\lambda_1(\triangle^{p})\right)
\le
\mathcal{B}_k\!\left(p;\lambda_1(\triangle^{p_{i,j}})\right).
\end{equation}
Since
\[
J_k(\triangle^p)=\mathcal{B}_k\!\left(p;\lambda_1(\triangle^{p})\right),
\]
it remains to estimate the two outer terms uniformly for $p\in\mathcal{C}_{ij}$.

For fixed $\Lambda\ge0$, set
\[
\widetilde{\mathcal{B}}_k(x,\theta;\Lambda)
:=
\mathcal{B}_k\!\left((x,x\tan\theta);\Lambda\right).
\]
We write
\[
t=\tan\theta,\qquad s=\sec\theta,\qquad
a=xs,\qquad b=\sqrt{1-2x+x^2s^2},\qquad P=1+a+b.
\]
On $\Omega_{\mathrm{mid}}$ we have $t>0$, $a\le1$, and $b\le1$. Moreover,
\[
P_x=s+\frac{xs^2-1}{b},
\qquad
b(P-2xP_x)=(1-a)P,
\]
and hence $P-2xP_x\ge0$. Thus
\begin{align*}
\frac{\partial}{\partial x}\widetilde{\mathcal{B}}_1(x,\theta;\Lambda)
&=
\frac{\Lambda t}{2}
+\frac{\pi^2}{8}\frac{P}{x^2t}(P-2xP_x)
\ge0.
\end{align*}
Letting
\[
C_0:=\frac{4\pi^2}{\left(3+\sqrt{\pi\sqrt3}\right)^2},
\qquad
H:=P+\sqrt{2\pi xt},
\]
we have
\begin{align*}
\frac{\partial}{\partial x}\widetilde{\mathcal{B}}_2(x,\theta;\Lambda)
&=
\frac{\Lambda t}{2}
+C_0\frac{H}{2x^2t}(P-2xP_x)
\ge0.
\end{align*}

It remains to check monotonicity with respect to $\theta$. Since
\[
P_t=\frac{xt}{s}+\frac{x^2t}{b},
\qquad
D:=P-2tP_t=P-2x^2t^2\left(\frac1a+\frac1b\right),
\]
we first show that $D\ge0$. Put $u=a+b$ and $v=ab$. By Heron's formula for the triangle with side lengths $1,a,b$,
\[
4x^2t^2=(u^2-1)(1-u^2+4v).
\]
Since $a,b\le1$, we have $1\le u\le2$ and
\[
u-1\le v\le\frac{u^2}{4}.
\]
The inequality $D\ge0$ is equivalent to
\[
E(u,v):=(1+u)v-\frac{u}{2}(u^2-1)(1-u^2+4v)\ge0.
\]
The function $E$ is affine in $v$, and at the endpoints one has
\[
E(u,u-1)
=
\frac{(u-2)(u-1)(u+1)(u^2-2u-1)}{2}
\ge0,
\]
\[
E\left(u,\frac{u^2}{4}\right)
=
-\frac{u(u-2)(u+1)}{4}
\ge0
\]
for $1\le u\le2$. Hence $D\ge0$. Consequently,
\begin{align*}
\frac{\partial}{\partial\theta}\widetilde{\mathcal{B}}_1(x,\theta;\Lambda)
&=
(1+t^2)
\left[
\frac{\Lambda x}{2}
+\frac{\pi^2}{8}\frac{PD}{xt^2}
\right]
\ge0,\\
\frac{\partial}{\partial\theta}\widetilde{\mathcal{B}}_2(x,\theta;\Lambda)
&=
(1+t^2)
\left[
\frac{\Lambda x}{2}
+C_0\frac{HD}{2xt^2}
\right]
\ge0.
\end{align*}

Therefore, for each fixed $\Lambda\ge0$, the function
$(x,\theta)\mapsto\widetilde{\mathcal{B}}_k(x,\theta;\Lambda)$ is increasing in both variables on each verification cell. Hence, for every $p\in\mathcal{C}_{ij}$,
\begin{equation}\label{eq:Bk-geom-bounds}
\mathcal{B}_k(p_{i,j};\Lambda)
\le
\mathcal{B}_k(p;\Lambda)
\le
\mathcal{B}_k(p_{i+1,j+1};\Lambda).
\end{equation}
Applying \eqref{eq:Bk-geom-bounds} with
$\Lambda=\lambda_1(\triangle^{p_{i+1,j+1}})$ to the left-hand side of
\eqref{eq:Bk-eig-sandwich}, and with
$\Lambda=\lambda_1(\triangle^{p_{i,j}})$ to the right-hand side of
\eqref{eq:Bk-eig-sandwich}, we obtain
\[
\mathcal{B}_k\!\left(p_{i,j};\,\lambda_1(\triangle^{p_{i+1,j+1}})\right)
\le
\mathcal{B}_k\!\left(p;\lambda_1(\triangle^{p})\right)
\le
\mathcal{B}_k\!\left(p_{i+1,j+1};\,\lambda_1(\triangle^{p_{i,j}})\right),
\]
which is exactly \eqref{eq:bounds-over-Cij-Jk}.
\end{proof}

To compute the uniform bounds of $J$ in each $\mathcal{C}_{ij}$, we employ the following algorithm, which implements the estimates derived in Lemma \ref{lem:bounds-over-Rij}:

\begin{framed}
  \captionsetup{type=algorithm}
  \captionof{algorithm}{Verified bounds of $J_i(\triangle^{p})~~(i=1,2)$ over the cell $\mathcal{C}_{ij}$}
  \label{algorithm-3}

  \begin{algorithmic}
    \State \textbf{Data:} Verification cell $\mathcal{C}_{ij}$ with vertices $\{p_{i,j}, p_{i+1,j}, p_{i,j+1}, p_{i+1,j+1}\}$.
    \State \textbf{Result:} A lower bound $L_{ij}$ of $J_1(\triangle^p),J_2(\triangle^p)$ for all $p \in \mathcal{C}_{ij}$.
    \State
    \State \textbf{Eigenvalue Estimation}
    \State Compute the rigorous lower bound of the eigenvalue at $\triangle^{p_{i+1,j+1}}$ and the upper bound at the triangle $\triangle^{p_{i,j}}$ using Lemma \ref{lem:est-tau}:
    \begin{equation}
        \lambda_{\min}^{\text{low}} := \underline{\lambda}_1(\triangle^{p_{i+1,j+1}}).
    \end{equation}
    \State
    \State \textbf{Functional Evaluation}
    \State Compute the uniform bound for $J_i~~(i=1,2)$ using the auxiliary functional $\mathcal{B}$:
    \begin{align}
        L_{ij} &= \mathcal{B}(p_{i,j}; \lambda_{\min}^{\text{low}}).
    \end{align}
  \end{algorithmic}
\end{framed}

As a result of computation, we obtain the following lower bound:
\begin{equation*}
J_1(\triangle^p)\geq 6.85\cdot 10^{-7},~~J_2(\triangle^p)\geq1.23\cdot 10^{-5}~~\forall p\in \Omega_{\mathrm{mid}}.
\end{equation*}

\subsection{Step 3: Case of $\Omega_{\mathrm{down}}$}\label{sec:case3}

In this subsection, we show that neither $J_1$ nor $J_2$ attains its global minimum in the degenerate region $\Omega_{\mathrm{down}}$.

\medskip\medskip

\textbf{Positivity of $J_1$}

To show the positivity of $J_1$ over $\Omega_{\mathrm{down}}$, let us recall a lower bound for the Dirichlet eigenvalues of thin triangles derived in \cite{Endo-Liu-2025-JDE-degenerate}.

\begin{theorem}[Theorem 3.5 of \cite{Endo-Liu-2025-JDE-degenerate}]\label{Thm:Airy-bound}
Let $\widetilde{\triangle}^{(s,t)}$ be the triangle with vertices $(-1,0)$, $(1,0)$, and $(s,t)$.
For any $t \in (0,t_0]$, the $k$-th Dirichlet eigenvalue satisfies
\begin{equation}\label{eq:Airy-general-bound}
  t^{4/3}\left(\lambda_k\bigl(\widetilde{\triangle}^{(s,t)}\bigr)-\frac{\pi^2}{t^2}\right)
  \ge
  \frac{(2\pi^2)^{2/3}\kappa_k(s)}
  {1+\frac{t_0^{2/3}}{3\pi^2}(2\pi^2)^{2/3}\kappa_k(s)},
\end{equation}
where $\kappa_k(s)$ is the $k$-th positive solution to the implicit equation
$f_s(\kappa)=0$ defined by
\begin{align}\label{eq:airy-eigen}
  f_s(\kappa)
  &:= \sqrt[3]{1+s}\,\mathcal{A}\bigl((1+s)^{2/3}\kappa\bigr)\,
      \mathcal{A}'\bigl((1-s)^{2/3}\kappa\bigr) \\
  &\quad
      + \sqrt[3]{1-s}\,\mathcal{A}\bigl((1-s)^{2/3}\kappa\bigr)\,
      \mathcal{A}'\bigl((1+s)^{2/3}\kappa\bigr). \nonumber
\end{align}
Here, $\mathcal{A}(u):=\mathrm{Ai}(-u)$ is the reversed Airy function.
\end{theorem}

To obtain a uniform lower bound independent of the horizontal parameter $s$, we establish the following property of $\kappa_1(s)$.

\begin{lemma}\label{lem:kappa-bound}
The smallest positive solution $\kappa_1(s)$ of \eqref{eq:airy-eigen} satisfies
\[
  \kappa_1(s)>1
  \qquad \text{for all } s\in[0,1).
\]
\end{lemma}

\begin{proof}
Let
\[
  \widetilde{\lambda}_1(s,t):=\lambda_1\bigl(\widetilde{\triangle}^{(s,t)}\bigr)
\]
denote the first Dirichlet eigenvalue of $-\Delta$ on the triangle
$\widetilde{\triangle}^{(s,t)}$.
By Steiner symmetrization with respect to the axis $x=0$, this triangle is transformed
into the isosceles triangle with vertices $(-1,0)$, $(1,0)$, and $(0,t)$, preserving
the base and the height. Since Steiner symmetrization does not increase the first
Dirichlet eigenvalue (see, e.g., \cite{Polya-Szego-1951}), we have
\begin{equation}\label{eq:Steiner-mu}
  \widetilde{\lambda}_1(s,t)\ge \widetilde{\lambda}_1(0,t)
  \qquad\text{for all } s\in(-1,1),\ t>0.
\end{equation}
Since $\kappa_1(s)$ is even in $s$, it is enough to consider $s\in[0,1)$.

Fix $s\in[0,1)$, and choose $s_0$ so that $s<s_0<1$. By
\cite[Theorem~1.2]{ourmieres2015dirichlet}, for $n=1$ there exist coefficients
$\beta_{j,1}(s)$ such that
\[
  \widetilde{\lambda}_1(s,t)\sim_{t\to0} t^{-2}\sum_{j\ge0}\beta_{j,1}(s)\,t^{j/3}
\]
uniformly for $s\in[-s_0,s_0]$. In particular,
\[
  \beta_{0,1}(s)=\pi^2,\qquad
  \beta_{1,1}(s)=0,\qquad
  \beta_{2,1}(s)=(2\pi^2)^{2/3}\kappa_1(s).
\]
Hence there exist constants $C>0$ and $t_0'>0$ such that, for all
$s\in[-s_0,s_0]$ and $0<t<t_0'$,
\begin{equation}\label{eq:mu-expansion-unif}
  \widetilde{\lambda}_1(s,t)
  = \frac{\pi^2}{t^2}
    + (2\pi^2)^{2/3}\kappa_1(s)\,t^{-4/3}
    + R(s,t),
  \qquad |R(s,t)|\le Ct^{-1}.
\end{equation}

Applying \eqref{eq:mu-expansion-unif} with $s$ and with $0$, and subtracting, we obtain
\[
  \widetilde{\lambda}_1(s,t)-\widetilde{\lambda}_1(0,t)
  = (2\pi^2)^{2/3}\bigl(\kappa_1(s)-\kappa_1(0)\bigr)t^{-4/3}
    + R(s,t)-R(0,t).
\]
Multiplying by $t^{4/3}$ yields
\[
  t^{4/3}\bigl(\widetilde{\lambda}_1(s,t)-\widetilde{\lambda}_1(0,t)\bigr)
  = (2\pi^2)^{2/3}\bigl(\kappa_1(s)-\kappa_1(0)\bigr)
    + t^{4/3}\bigl(R(s,t)-R(0,t)\bigr).
\]
By \eqref{eq:mu-expansion-unif}, the last term tends to $0$ as $t\to0+$. On the other hand,
\eqref{eq:Steiner-mu} shows that the left-hand side is nonnegative for all $t>0$. Therefore,
letting $t\to0+$, we obtain
\[
  (2\pi^2)^{2/3}\bigl(\kappa_1(s)-\kappa_1(0)\bigr)\ge0,
\]
and hence $\kappa_1(s)\ge\kappa_1(0)$ for all $s\in[0,1)$.

Finally, by \cite[Remark~2.2]{ourmieres2015dirichlet},
\[
  \kappa_1(0)=z_A^0(1)>1,
\]
where $z_A^0(1)$ denotes the first positive zero of $A'(u)$, with
$A(u)=\mathrm{Ai}(-u)$. Combined with $\kappa_1(s)\ge\kappa_1(0)$, this yields
$\kappa_1(s)>1$ for all $s\in[0,1)$.
\end{proof}

Using Lemma \ref{lem:kappa-bound}, we derive an explicit lower bound for $\lambda_1(\triangle^p)$ in terms of the height $y$, where $p=(x,y)$.

\begin{lemma}\label{lem:lambda-explicit-bound}
Let $\triangle^p$ be a triangle with $p=(x,y)$ and $y\in(0,t_0/2]$. Then
\begin{equation}\label{eq:lambda-explicit}
  \lambda_1(\triangle^p)\ge \frac{\pi^2}{y^2}
  + \nu_{t_0}\frac{2^{4/3}\pi^{4/3}}{y^{4/3}},
\end{equation}
where
\[
  \nu_{t_0}:=\left(1+\frac{(2\pi^2t_0)^{2/3}}{3\pi^2}\right)^{-1}.
\]
\end{lemma}

\begin{proof}
Since
\[
  \triangle^p=\frac12\,\widetilde{\triangle}^{(2x-1,\,2y)},
\]
we have
\[
  \lambda_1(\triangle^p)
  =4\,\lambda_1\bigl(\widetilde{\triangle}^{(2x-1,\,2y)}\bigr).
\]
Applying \eqref{eq:Airy-general-bound} with $s=2x-1$ and $t=2y$, we obtain
\begin{equation}\label{eq:lambda-Preliminaries}
  \lambda_1(\triangle^p)
  \ge \frac{\pi^2}{y^2}
  + 2^{4/3}\pi^{4/3}
    \frac{\kappa_1(2x-1)}{1+C_{t_0}\kappa_1(2x-1)}
    \frac{1}{y^{4/3}},
\end{equation}
where
\[
  C_{t_0}:=\frac{t_0^{2/3}}{3\pi^2}(2\pi^2)^{2/3}
  =\frac{(2\pi^2t_0)^{2/3}}{3\pi^2}.
\]
Since the function
\[
  u\mapsto \frac{u}{1+C_{t_0}u}
\]
is strictly increasing on $(0,\infty)$, Lemma \ref{lem:kappa-bound} yields
\begin{equation}\label{eq:kappa-ratio-lower}
  \frac{\kappa_1(2x-1)}{1+C_{t_0}\kappa_1(2x-1)}
  \ge \frac{1}{1+C_{t_0}}
  = \nu_{t_0}.
\end{equation}
Substituting \eqref{eq:kappa-ratio-lower} into \eqref{eq:lambda-Preliminaries}, we obtain
\eqref{eq:lambda-explicit}.
\end{proof}

\begin{lemma}\label{lem:J1-bound-down}
For every $p \in \Omega_{\mathrm{down}}$, we have $J_1(\triangle^p)>0$.
\end{lemma}

\begin{proof}
Let $p=(x,y)\in \Omega_{\mathrm{down}}$. We first estimate the perimeter:
\[
  |\partial\triangle^p|
  =1+\sqrt{x^2+y^2}+\sqrt{(1-x)^2+y^2}.
\]
For fixed $y>0$, the function
\[
  x\mapsto \sqrt{x^2+y^2}+\sqrt{(1-x)^2+y^2}
\]
is increasing on $[1/2,1]$. Hence
\[
  |\partial\triangle^p|
  \le 1+y+\sqrt{1+y^2}
  < 2+y+\frac{y^2}{2},
\]
where we used $\sqrt{1+y^2}<1+\frac{y^2}{2}$. Since $|\triangle^p|=y/2$, it follows that
\[
  \frac{|\partial\triangle^p|^2}{|\triangle^p|}
  < \frac{2}{y}\left(2+y+\frac{y^2}{2}\right)^2
  = \frac{8}{y}+8+6y+2y^2+\frac{y^3}{2}.
\]
Therefore,
\begin{equation}\label{eq:perimeter-term-down}
  -\frac{\pi^2}{16}\frac{|\partial\triangle^p|^2}{|\triangle^p|}
  > -\frac{\pi^2}{2y}-\frac{\pi^2}{2}-\frac{3\pi^2}{8}y
    -\frac{\pi^2}{8}y^2-\frac{\pi^2}{32}y^3.
\end{equation}

On the other hand, by Lemma \ref{lem:lambda-explicit-bound} with $t_0=0.08$,
\begin{equation}\label{eq:eigenvalue-term-down}
  \lambda_1(\triangle^p)|\triangle^p|
  \ge \left(\frac{\pi^2}{y^2}+\nu_{t_0}\frac{2^{4/3}\pi^{4/3}}{y^{4/3}}\right)\frac{y}{2}
  = \frac{\pi^2}{2y}+\nu_{t_0}\pi^{4/3}2^{1/3}y^{-1/3}.
\end{equation}
Combining \eqref{eq:perimeter-term-down} and \eqref{eq:eigenvalue-term-down}, we obtain $J_1(\triangle^p)>\Phi(y)$,
where
\[
  \Phi(y):=
  \nu_{t_0}\pi^{4/3}2^{1/3}y^{-1/3}
  -\left(\frac{\pi^2}{2}+\frac{7\sqrt3\pi^2}{12}\right)
  -\frac{3\pi^2}{8}y
  -\frac{\pi^2}{8}y^2
  -\frac{\pi^2}{32}y^3.
\]
Since $\Phi$ is strictly decreasing on $(0,\infty)$ and $y\le 0.04$, it is enough to show that
$\Phi(0.04)>0$. A direct computation using $\nu_{t_0}>0.949$ gives
\[
  \Phi(0.04)
  > 0.949\,\pi^{4/3}2^{1/3}(0.04)^{-1/3}
    -\left(\frac{\pi^2}{2}+\frac{7\sqrt3\pi^2}{12}\right)
    -\frac{3\pi^2}{8}(0.04)
    -\frac{\pi^2}{8}(0.04)^2
    -\frac{\pi^2}{32}(0.04)^3
  > 0.
\]
Therefore $J_1(\triangle^p)>0$.
\end{proof}

\medskip

\textbf{Positivity of $J_2$}

The following lemma shows the positivity of $J_2$ over $\Omega_{\mathrm{down}}$.

\begin{lemma}\label{lem:J2-bound-down}
We have that $J_{2}\left(\triangle^p\right)>0$ for all  $p=(x,y)~(\in\Omega_{\mathrm{down}})$ with $y\in\left(0,0.11\right)$.
\end{lemma}

\begin{proof}
By Corollary~4.1 of Freitas and Siudeja \cite{Freitas-Siudeja-2010}, we have
\[
\lambda_{1}\left(\triangle^p\right)\geq\frac{\pi^{2}}{4\left|\triangle^p\right|^{2}}\left(d\left(\triangle^p\right)+\frac{2\left|\triangle^p\right|}{d\left(\triangle^p\right)}\right)^{2}=\frac{\pi^{2}(1+y)^{2}}{y^{2}}
\]
so that
\[
\lambda_{1}\left(\triangle^p\right)\left|\triangle^p\right|\geq\frac{\pi^{2}(1+y)^{2}}{2y},
\]
hence
\begin{align*}
J_{2}\left(\triangle^p\right) & =\lambda_{1}\left(\triangle^p\right)\left|\triangle^p\right|-\frac{\pi^{2}}{\left(3+\sqrt{\pi\sqrt{3}}\right)^{2}}\frac{\left(\left|\partial\triangle^p\right|+\sqrt{4\pi\left|\triangle^p\right|}\right)^{2}}{\left|\triangle^p\right|}\\
 & \ge\frac{\pi^{2}(1+y)^{2}}{2y}-\frac{\pi^{2}}{\left(3+\sqrt{\pi\sqrt{3}}\right)^{2}}\frac{\left(\left|\partial\triangle^p\right|+\sqrt{4\pi\left|\triangle^p\right|}\right)^{2}}{\left|\triangle^p\right|}.
\end{align*}

Since $\left|\partial\triangle^p\right|$ is maximized for the
the isosceles triangle $x=\sqrt{1-y^{2}}$ plugging this in means
\[
\left|\partial\triangle^p\right|\leq\left|\partial\triangle^{\sqrt{1-y^{2}},y}\right|=2+\sqrt{2-2\sqrt{1-y^{2}}}.
\]
Now
\begin{align*}
J_{2}\left(\triangle^p\right) & =\lambda_{1}\left(\triangle^p\right)\left|\triangle^p\right|-\frac{\pi^{2}}{\left(3+\sqrt{\pi\sqrt{3}}\right)^{2}}\frac{\left(\left|\partial\triangle^p\right|+\sqrt{4\pi\left|\triangle^p\right|}\right)^{2}}{\left|\triangle^p\right|}\\
 & \ge\frac{\pi^{2}(1+y)^{2}}{2y}-\frac{2\pi^{2}}{\left(3+\sqrt{\pi\sqrt{3}}\right)^{2}}\frac{\left(2+\sqrt{2-2\sqrt{1-y^{2}}}+\sqrt{2\pi y}\right)^{2}}{y}
\end{align*}
Consider the function
\[
\phi\left(y\right)=\frac{\pi^{2}(1+y)^{2}}{2y}-\frac{2\pi^{2}}{\left(3+\sqrt{\pi\sqrt{3}}\right)^{2}}\frac{\left(2+\sqrt{2-2\sqrt{1-y^{2}}}+\sqrt{2\pi y}\right)^{2}}{y}
\]
We aim to find $y_{0}$ so that
\[
\phi\left(y\right)\geq0,\forall y\in(0,y_{0}).
\]
First, note that we can just consider
\begin{align*}
\phi\left(y\right) & =\frac{\pi^{2}(1+y)^{2}}{2y}-\frac{2\pi^{2}}{\left(3+\sqrt{\pi\sqrt{3}}\right)^{2}}\frac{\left(2+\sqrt{2-2\sqrt{1-y^{2}}}+\sqrt{2\pi y}\right)^{2}}{y}\\
 & =\frac{\pi^{2}}{y}\left(\frac{1}{2}\left(1+y\right)^{2}-\frac{2}{\left(3+\sqrt{\pi\sqrt{3}}\right)^{2}}\left(2+\sqrt{2-2\sqrt{1-y^{2}}}+\sqrt{2\pi y}\right)^{2}\right)\\
 & =:\frac{\pi^{2}}{y}F\left(y\right)
\end{align*}
where
\begin{align*}
F\left(y\right) & =\frac{1}{2}\left(1+y\right)^{2}-\frac{2}{\left(3+\sqrt{\pi\sqrt{3}}\right)^{2}}\left(2+\sqrt{2-2\sqrt{1-y^{2}}}+\sqrt{2\pi y}\right)^{2}\\
 & =\frac{1}{2}\left(1+y\right)^{2}-\frac{2}{\left(3+\sqrt{\pi\sqrt{3}}\right)^{2}}G\left(y\right)^{2}.
\end{align*}

First we show $F$ strictly decreases on $(0,0.2]$. To see this, we compute the derivative of $F$ to obtain
\[
F^{\prime}\left(y\right)=1+y-\frac{4}{\left(3+\sqrt{\pi\sqrt{3}}\right)^{2}}G\left(y\right)G^{\prime}\left(y\right).
\]
Note that on $\left(0,1\right),$ the mapping $y\mapsto2-2\sqrt{1-y^{2}}$
is increasing, implying that $y\mapsto\sqrt{2-2\sqrt{1-y^{2}}}$ is
increasing with non-negative derivative, which implies
\begin{align*}
G^{\prime}\left(y\right) & =\frac{d}{dy}\left(\sqrt{2-2\sqrt{1-y^{2}}}\right)+\frac{d}{dy}\sqrt{2\pi y}\geq0+\sqrt{\frac{\pi}{2y}}.
\end{align*}
We also clearly have that $G\left(y\right)\geq2+\sqrt{2\pi y}$. Using
both of these facts imply that 
\begin{align*}
F^{\prime}\left(y\right) & =1+y-\frac{4}{\left(3+\sqrt{\pi\sqrt{3}}\right)^{2}}G\left(y\right)G^{\prime}\left(y\right)\\
 & \leq1+y-\frac{4}{\left(3+\sqrt{\pi\sqrt{3}}\right)^{2}}\left(2+\sqrt{2\pi y}\right)\sqrt{\frac{\pi}{2y}}\\
 & =1+y-\frac{4}{\left(3+\sqrt{\pi\sqrt{3}}\right)^{2}}\left(2\sqrt{\frac{\pi}{2y}}+\pi\right):=-h\left(y\right).
\end{align*}
We can compute
\[
h^{\prime}\left(y\right)=-\frac{4}{\left(3+\sqrt{\pi\sqrt{3}}\right)^{2}}\sqrt{\frac{\pi}{2}}y^{-\frac{3}{2}}-1<0.
\]
This shows $h$ is strictly decreasing on $(0,0.2]$. Hence $h\left(y\right)\geq h\left(0.2\right)\geq0.03$
for all $y\in(0,0.2]$. This implies $F^{\prime}\left(y\right)\leq-h\left(y\right)<0$
for all $y\in(0,0.2]$. This proves $F$ is strictly decreasing on
$(0,0.2]$.

Direct computation shows that 
\[
F\left(0.11\right)\geq0.0075>0
\]
and using the fact that $F$ is decreasing and continuous shows that
$F\left(y\right)>0$ for all $y\in(0,0.11]$.
\end{proof}

\section*{Acknowledgement}

The first and the second authors are supported by Japan Society for the Promotion of Science. The first author is supported by JSPS KAKENHI Grant Number JP24KJ1170. The second author is supported by JSPS KAKENHI Grant Numbers JP24K00538 and 24K21314. The third author is supported in part by  NSF Grant DMS-2316968. 

The authors would like to thank Ilias Ftouhi for helpful discussions. 

\bibliographystyle{plain}
\bibliography{EndoRef}

\end{document}